\documentclass[12pt, a4paper]{article}
\usepackage{amsmath}
\usepackage{amsfonts,amssymb}

\usepackage{times}
\usepackage[T1]{fontenc}
\usepackage[latin1]{inputenc}
\usepackage{authblk}
\usepackage{geometry}
\usepackage{amssymb}
\usepackage{amsthm}

\makeatletter

\DeclareMathOperator{\Gal}{Gal}

\makeatother
 \newtheorem{thm}{Theorem}%% [section]
 \newtheorem{cor}%%[thm]
 {Corollary}
 \newtheorem{lem}%%[thm]
 {Lemma}
 \newtheorem{prop}%% [thm]
 {Proposition}
 \newtheorem{defn}%%[thm]
 {Definition}
 \newtheorem{rem}%%[thm]
 {Remark}
\newtheorem{ex}%% [thm]
{Example}
\newtheorem{conj}%%[thm]
{Conjecture}
\numberwithin{equation}{section}
\title{Classification of quantum groups and Belavin--Drinfeld cohomologies}

\author[*]{Boris Kadets} 
\author[*]{Eugene Karolinsky}
\author[**]{Iulia Pop}
\author[**]{Alexander Stolin}
\affil[*]{Department of Mechanics and Mathematics, Kharkov National University}
\affil[**]{ Department of Mathematics, Gothenburg University, Sweden}

\begin{document}
\maketitle

\begin{abstract}
In the present article we discuss the classification of quantum groups whose quasi-classical limit is a given simple complex Lie algebra $\mathfrak{g}$. This problem reduces to the classification of all Lie bialgebra
structures on $\mathfrak{g}(\mathbb{K})$, where $\mathbb{K}=\mathbb{C}((\hbar))$. The associated classical double
is of the form $\mathfrak{g}(\mathbb{K})\otimes_{\mathbb{K}} A$, where $A$ is one of the following: $\mathbb{K}[\varepsilon]$, where $\varepsilon^{2}=0$, $\mathbb{K}\oplus\mathbb{K}$ or $\mathbb{K}[j]$, where $j^{2}=\hbar$. The first case relates to quasi-Frobenius Lie algebras. In the second and third cases we introduce a theory of Belavin--Drinfeld cohomology associated to any non-skewsymmetric $r$-matrix from the Belavin--Drinfeld list \cite{BD}. We prove a one-to-one correspondence between gauge equivalence classes of Lie bialgebra structures on $\mathfrak{g}(\mathbb{K})$ and cohomology classes (in case II) and twisted cohomology classes (in case III) associated to any non-skewsymmetric $r$-matrix.


\end{abstract}
\section{Introduction}
The first example of a quantum group appeared in the work \cite{KR} (see \cite{KRe} for English translation), where the quantum group $U_{\hslash}(\mathfrak{sl}_2)$ was constructed. Later, this example was generalized in the works of Jimbo {\cite{Jim}} and Drinfeld \cite{D} (see also the references therein). The aim of this paper is to present an approach to the general problem of classification of quantum groups lying over a simple finite-dimensional complex Lie algebra.  

Let $k$ be a field of characteristic 0. According to \cite{D}, a quantized universal enveloping algebra (or a quantum group) is a topologically free topological Hopf algebra $H$ over the formal power series ring
$k[[\hbar]]$ such that $H/\hbar H$ is isomorphic to the universal enveloping algebra of a Lie algebra $\mathfrak{g}$ over $k$.

The quasi-classical limit of a quantum group is a Lie bialgebra. By definition, a Lie bialgebra is a Lie algebra $\mathfrak{g}$ together with a cobracket
$\delta$ which is compatible with the Lie bracket. Given a quantum group $H$,
with comultiplication $\Delta$, the quasi-classical limit of $H$ is the Lie bialgebra $\mathfrak{g}$ of primitive elements of $H/\hbar H$ and the cobracket is the restriction of the map $(\Delta-\Delta^{21})/\hbar\ (\!\!\!\!\mod\hbar)$ to
$\mathfrak{g}$.

The operation of taking the semiclassical limit is a functor
$SC: QUE\rightarrow LBA$ between categories of quantum groups and Lie bialgebras over $k$. The quantization problem raised by Drinfeld aims at finding a quantization functor, i.\ e., a functor $Q: LBA \rightarrow QUE$ such that $SC\circ Q$ is isomorphic to the identity. Moreover, a quantization functor is required to be universal, in the sense of props.

The existence of universal quantization functors was proved by Etingof and Kazhdan \cite{EK1}, \cite{EK2}. They used Drinfeld's theory of associators to construct quantization functors for any field $k$ of characteristic zero. Drinfeld introduced the notion of associator in relation to the theory of quasi-triangular
quasi-Hopf algebras and showed that associators exist over any field
$k$ of characteristic zero. Etingof and Kazhdan proved that for any fixed associator over $k$ one can construct a universal quantization
functor. More precisely, let $(\mathfrak{g},\delta)$ be a Lie bialgebra over $k$. Then one can associate a Lie bialgebra $\mathfrak{g}_{\hbar}$ over $k[[\hbar]]$
defined as $(\mathfrak{g}\otimes_{k}k[[\hbar]], \hbar\delta)$.
According to Theorem 2.1 of \cite{EK2}
there exists an equivalence $\widehat{Q}$
between the category $LBA_{0}(k[[\hbar]])$
of topologically free over $k[[\hbar]]$ Lie bialgebras
with $\delta=0\ (\!\!\!\!\mod\hbar)$ and the category $HA_{0}(k[[\hbar]])$ of
topologically free Hopf algebras cocommutative modulo $\hbar$.
Moreover, for any $(\mathfrak{g},\delta)$ over $k$, one has the following:
$\widehat{Q}(\mathfrak{g}_{\hbar})=U_{\hbar}(\mathfrak{g})$.

The aim of the present article
is the classification of quantum groups whose quasi-classical limit is a given simple complex Lie algebra $\mathfrak{g}$. Due to the equivalence between $HA_{0}(\mathbb{C}[[\hbar]])$ and $LBA_{0}(\mathbb{C}[[\hbar]])$, this problem is equivalent to classification of Lie bialgebra structures on
$\mathfrak{g}\otimes_{\mathbb{C}}\mathbb{C}[[\hbar]]$. For simplicity, denote $\mathbb{O}:=\mathbb{C}[[\hbar]]$, $\mathbb{K}:=\mathbb{C}((\hbar))$, $\mathfrak{g}(\mathbb{O}):=\mathfrak{g}\otimes_{\mathbb{C}}\mathbb{O}$ and $\mathfrak{g}(\mathbb{K}):= \mathfrak{g}\otimes_{\mathbb{C}} \mathbb{K}$.

On the other hand, in order to classify cobrackets on $\mathfrak{g}(\mathbb{O})$ it is enough to classify cobrackets on $\mathfrak{g}(\mathbb{K})$. Indeed,
if $\delta$ is a Lie bialgebra structure on $\mathfrak{g}(\mathbb{O})$, then it can be naturally extended to $\mathfrak{g}(\mathbb{K})$. Conversely, given
a Lie bialgebra
structure $\bar{\delta}$ on $\mathfrak{g}(\mathbb{K})$, then by multiplying $\bar{\delta}$ by an
appropriate power of $\hbar$, the restriction of $\bar{\delta}$ to
$\mathfrak{g}(\mathbb{O})$ is a Lie bialgebra structure on $\mathfrak{g}(\mathbb{O})$.

From now on let $G$ be a connected algebraic group with a reductive Lie algebra whose semisimple part is
$\mathfrak{g}$. We will consider the adjoint action $\mathrm{Ad}$ of $G$ on $\mathfrak{g}$. We consider the equivalence classes of Lie bialgebra structures on $\mathfrak{g}(\mathbb{K})$ with respect to the following equivalence: two bialgebra structures $\delta_1, \delta_2$ are equivalent, if there exists an element $a \in \mathbb{K}^*$ and $X \in G(\mathbb{K})$ such that $\delta_1=a (\mathrm{Ad}_X \otimes \mathrm{Ad}_X)\delta_2$. We will also use the term ``gauge equivalence'' or ``$G$-equivalence'' if there exists $X \in G(\mathbb{K})$ such that $\delta_1=(\mathrm{Ad}_X \otimes \mathrm{Ad}_X)\delta_2$.

From the general theory of Lie bialgebras it is known that for each Lie bialgebra structure $\delta$ on a fixed Lie algebra $L$ one can construct the corresponding classical double $D(L, \delta)$ which is the vector space $L\oplus L^{*}$ together with a bracket which is induced by the bracket and cobracket of $L$, and a non-degenerate invariant bilinear form, see  \cite{D1}. We consider $L=\mathfrak{g}(\mathbb{K})$ and prove Proposition \ref{prop1} which states that there exists
an associative, unital, commutative
algebra $A$, of dimension 2 over $\mathbb{K}$, such that
$D(\mathfrak{g}(\mathbb{K}), \delta)\cong \mathfrak{g}(\mathbb{K})\otimes_{\mathbb{K}} A$. In Proposition \ref{prop2} we show that there are three possibilities for $A$: $A=\mathbb{K}[\varepsilon]$, where $\varepsilon^{2}=0$, $A=\mathbb{K}\oplus\mathbb{K}$ or $A=\mathbb{K}[j]$, where $j^{2}=\hbar$.

Due to the correspondence Lie bialgebras -- Manin triples, to any Lie bialgebra structure $\delta$ on
$L$ one can associate a certain Lagrangian subalgebra $W$ of $D(L, \delta)$
which is complementary to $L$ and conversely, any such $W$ produces a Lie cobracket on $L$. The main problem is to obtain a classification of all such
subalgebras $W$ for the three choices of $A$ as above. We investigate
separately each choice of $A$.

For $A=\mathbb{K}[\varepsilon]$, where $\varepsilon^{2}=0$, it turns out that the classification problem is related to that of quasi-Frobenius Lie subalgebras over $\mathbb{K}$.

In the case of $A=\mathbb{K}\oplus\mathbb{K}$, we introduce Belavin--Drinfeld cohomologies.
Namely, for any non-skewsymmetric constant $r$-matrix $r_{BD}$ from the Belavin--Drinfeld list \cite{BD}, we associate a cohomology set $H^{1}_{BD}(r_{BD})$.
This cohomology set will depend on a gauge group $G$ acting ``naturally'' on $\mathfrak{g}$. We will see that the choice of $G$ matters.
Therefore, we will use notation $H^{1}_{BD}(G, r_{BD})$.
It is worthful to notice that in all the cases with exception for $GL(n)$, the Lie algebra of $G$ will be $\mathfrak{g}$.

We prove that there exists a one-to-one correspondence between any Belavin--Drinfeld cohomology and gauge equivalence classes of
Lie bialgebra structures on $\mathfrak{g}(\mathbb{K})$.
Then we restrict our discussion to  $\mathfrak{g}=sl(n)$ and we show that all cohomologies $H^{1}_{BD}( GL(n), r_{BD})$ are trivial.
%However, we will show that $H^{1}_{BD}(SL(n), r_{BD})$ is not trivial for certain $r_{BD}$.

We also discuss the case of orthogonal algebras $\mathfrak{g}=o(n)$, where it turns out that the cohomology associated to the
Drinfeld--Jimbo $r$-matrix is also trivial.
We also illustrate an example where the cohomology corresponding to another non-skewsymmetric constant $r$-matrix
for $o(2n)$ is non-trivial.

We finally move to the classification of Lie bialgebras whose classical
double is isomorphic to $\mathfrak{g}(\mathbb{K}[j])$, with $j^{2}=\hbar$. We restrict ourselves to $\mathfrak{g}=sl(n)$
and we show that in this case a cohomology theory can be introduced too.
Our result states that there exists a one-to-one correspondence between Belavin--Drinfeld twisted cohomology
and gauge equivalence classes of Lie bialgebra structures on $\mathfrak{g}(\mathbb{K})$.
We prove that the twisted cohomology corresponding to the Drinfeld--Jimbo $r$-matrix
and another class of $r$-matrices (called generalized Cremmer--Gervais)
is trivial.

In the last section of the article we compute Belavin--Drinfeld cohomology in certain cases for $\mathfrak{g}=sl(n)$ and $G=SL(n)$. In particular, we show that $H^{1}_{BD}(SL(n), r_{BD})$ is not trivial for certain $r_{BD}$.
Finally, we formulate a conjecture
stating that the Belavin--Drinfeld cohomology associated to the
Drinfeld--Jimbo $r$-matrix is trivial for any simple complex Lie algebra $\mathfrak{g}$.
We also define the quantum Belavin--Drinfeld cohomology and formulate a
second conjecture about the existence of a natural correspondence between classical and quantum cohomologies.

\section{ Lie bialgebra structures on $\mathfrak{g}(\mathbb{K})$ }

Let $\mathfrak{g}$ be a simple complex finite-dimensional Lie algebra.
Consider the Lie algebras $\mathfrak{g}(\mathbb{O})=\mathfrak{g}\otimes_{\mathbb{C}}\mathbb{O}$ and $\mathfrak{g}(\mathbb{K})= \mathfrak{g}\otimes _{\mathbb{C}}\mathbb{K}$.

We have seen that the classification of quantum groups with quasi-classical
limit $\mathfrak{g}$ is equivalent to the classification of all Lie bialgebra structures on $\mathfrak{g}(\mathbb{O})$. Moreover, as explained in the introduction, in order to classify Lie bialgebra structures
on $\mathfrak{g}(\mathbb{O})$, it is enough to classify them on $\mathfrak{g}(\mathbb{K})$.

Let us assume that $\bar{\delta}$ is a Lie bialgebra
structure on $\mathfrak{g}(\mathbb{K})$. This cobracket endows the dual of
$\mathfrak{g}(\mathbb{K})$ with a Lie bracket. Then one can construct the corresponding classical double $D(\mathfrak{g}(\mathbb{K}),\bar{\delta})$.
As a vector
space, $D(\mathfrak{g}(\mathbb{K}),\bar{\delta})=\mathfrak{g}(\mathbb{K})\oplus \mathfrak{g}(\mathbb{K})^{*}$. As a Lie algebra, it is endowed with a bracket which is induced by the bracket and cobracket of $\mathfrak{g}(\mathbb{K})$.
Moreover the canonical symmetric non-degenerate bilinear form on
this space is invariant.

Similarly to Lemma 2.1 from \cite{MSZ}, one can prove that
$D(\mathfrak{g}(\mathbb{K}),\bar{\delta})$ is a direct sum of regular adjoint $\mathfrak{g}$-modules.
Combining this result with Prop. 2.2 from \cite{BZ}, it follows that
\begin{prop}\label{prop1}
There exists an associative, unital, commutative algebra $A$, of dimension 2 over $\mathbb{K}$, such that
$D(\mathfrak{g}(\mathbb{K}),\bar{\delta})\cong \mathfrak{g}(\mathbb{K})\otimes_{\mathbb{K}} A$.
\end{prop}

\begin{rem}
The symmetric invariant non-degenerate bilinear form $Q$ on $\mathfrak{g}(\mathbb{K})\otimes_{\mathbb{K}} A$ is given in the following way.
For arbitrary elements $f_{1},f_{2}\in\mathfrak{g}(\mathbb{K})$ and $a,b\in A$ we have
$Q(f_{1}\otimes a,f_{2}\otimes b)=K(f_{1},f_{2})\cdot t(ab)$, where
$K$ denotes the Killing form on $\mathfrak{g}(\mathbb{K})$ and $t: A\longrightarrow \mathbb{K}$ is a trace function.
\end{rem}
Let us investigate the algebra $A$. Since $A$ is unital
and of dimension 2 over $\mathbb{K}$, one can choose a basis $\{e,1\}$, where $1$ denotes the unit.
Moreover, there exist $p$ and $q$ in $\mathbb{K}$ such that
$e^2+pe+q=0$. Let $\Delta=p^2-4q \in\mathbb {K}$. We distinguish the following cases:

(i) Assume $\Delta=0$. Let $\displaystyle \varepsilon:=e+\frac{p}{2}$.
Then $\varepsilon^2=0$
and $A=\mathbb{K}\varepsilon\oplus\mathbb{K}=\mathbb{K}[\varepsilon]$.

(ii) Assume $\Delta\neq 0$ and has even order as an element of $\mathbb{K}$.
This implies that $\Delta=\hbar^{2m}(a_0+a_1\hbar+a_2\hbar^{2}+\ldots )$, where $m$ is an integer, $a_i$ are complex coefficients and $a_0\neq 0$.

One can easily check that the equation $x^2=a_0+a_1\hbar+a_2\hbar^{2}+\ldots  $
has two solutions $\pm x=x_0+x_1\hbar+x_2\hbar^{2}+\ldots $ in $\mathbb{O}$.

Then $\displaystyle e=-\frac{p}{2}\pm \frac{\hbar^{m}x}{2}$, which implies that $e\in \mathbb{K}$ and $A=\mathbb{K}\oplus\mathbb{K}$.

(iii) Assume $\Delta\neq 0$ and has odd order as an element of $\mathbb{K}$.
We have $\Delta=\hbar^{2m+1}(a_0+a_1\hbar+a_2\hbar^{2}+\ldots )$, where $m$ is an integer, $a_i$ are complex coefficients and $a_0\neq 0$.

Again the equation $x^2=a_0+a_1\hbar+a_2\hbar^{2}+\ldots  $
has two solutions $\pm x=x_0+x_1\hbar+x_2\hbar^{2}+\ldots $ in $\mathbb{O}$. Since
$a_0\neq 0$, we have $x_0\neq 0$ and thus $x$ is invertible in $\mathbb{O}$.

Let $j=\hbar^{-m}(2e+p)x^{-1}$. Then $e^2+pe+q=0$ is equivalent to
$j^{2}=\hbar$.

On the other hand, $A=\mathbb{K}e\oplus \mathbb{K}$ and $2e=\hbar^{m}xj-p$
imply that $A=\mathbb{K}j\oplus \mathbb{K}$. Therefore, we obtain that
$A=\mathbb{K}[j]$ where $j^{2}=\hbar$.

We can summarize the above facts:
\begin{prop}\label{prop2}
Let $\bar{\delta}$ be an arbitrary Lie bialgebra structure on $\mathfrak{g}(\mathbb{K})$.
Then $D(\mathfrak{g}(\mathbb{K}),\bar{\delta})$ is isomorphic to $\mathfrak{g}(\mathbb{K})\otimes_{\mathbb{K}} A$,
where  $A=\mathbb{K}[\varepsilon]$ and $\varepsilon^{2}=0$, $A=\mathbb{K}\oplus\mathbb{K}$ or $A=\mathbb{K}[j]$ and $j^{2}=\hbar$.
\end{prop}

On the other hand, it is well-known, see for instance \cite{D},
that there is a one-to-one correspondence
between Lie bialgebra structures on a Lie algebra $L$ and Manin triples
$(D(L),L, W)$. For $L=\mathfrak{g}(\mathbb{K})$, this fact implies the following
\begin{prop}
There exists a one-to-one correspondence between
Lie bialgebra structures on $\mathfrak{g}(\mathbb{K})$ for which
the classical double is $\mathfrak{g}(\mathbb{K})\otimes_{\mathbb{K}} A$ and
Lagrangian subalgebras $W$ of $\mathfrak{g}(\mathbb{K})\otimes_{\mathbb{K}} A$, with respect to the non-degenerate bilinear form $Q$,
and transversal to $\mathfrak{g}(\mathbb{K})$.
\end{prop}

\begin{cor}
(i) There exists a one-to-one correspondence between
Lie bialgebra structures on $\mathfrak{g}(\mathbb{K})$ for which
the classical double is
$\mathfrak{g}(\mathbb{K}[\varepsilon])$,
$\varepsilon^{2}=0$, and
Lagrangian subalgebras $W$ of $\mathfrak{g}(\mathbb{K}[\varepsilon])$,
and transversal to $\mathfrak{g}(\mathbb{K})$.

(ii) There exists a one-to-one correspondence between
Lie bialgebra structures on $\mathfrak{g}(\mathbb{K})$ for which
the classical double is
$\mathfrak{g}(\mathbb{K})\oplus\mathfrak{g}(\mathbb{K})$ and
Lagrangian subalgebras $W$ of $\mathfrak{g}(\mathbb{K})\oplus\mathfrak{g}(\mathbb{K})$,
and transversal to $\mathfrak{g}(\mathbb{K})$, embedded diagonally into
$\mathfrak{g}(\mathbb{K})\oplus\mathfrak{g}(\mathbb{K})$.

(iii) There exists a one-to-one correspondence between
Lie bialgebra structures on $\mathfrak{g}(\mathbb{K})$ for which
the classical double is
$\mathfrak{g}(\mathbb{K}[j])$, where $j^{2}=\hbar$, and
Lagrangian subalgebras $W$ of $\mathfrak{g}(\mathbb{K}[j])$, and transversal to $\mathfrak{g}(\mathbb{K})$.

\end{cor}

\section{Lie bialgebra structures in Case I}

Here we study the Lie bialgebra structures $\delta$ on $\mathfrak{g}(\mathbb{K})$
for which the corresponding Drinfeld double is isomorphic to
$\mathfrak{g}(\mathbb{K}[\varepsilon])$,
$\varepsilon^{2}=0$. Our problem is to find all subalgebras
$W$ of
$\mathfrak{g}(\mathbb{K}[\varepsilon])$ satisfying the
following conditions:

(i) $W\oplus \mathfrak{g}(\mathbb{K})=\mathfrak{g}(\mathbb{K}[\varepsilon])$.

(ii) $W=W^{\perp}$ with respect to the non-degenerate symmetric bilinear form $Q$ on $\mathfrak{g}(\mathbb{K}[\varepsilon])$ given by

$$Q(f_1+\varepsilon f_2, g_1+\varepsilon g_2)=
K(f_1,g_2)+K(f_2,g_1).$$

\begin{prop}
Any subalgebra $W$ of
$\mathfrak{g}(\mathbb{K}[\varepsilon])$ satisfying  conditions (i) and (ii) from above is uniquely defined by a subalgebra $L$ of $\mathfrak{g}(\mathbb{K})$ together
with a non-degenerate 2-cocycle $B$ on $L$.
\end{prop}

\begin{proof}
The proof is similar to that of Th. 3.2 and Cor. 3.3 from \cite{S2}.
\end{proof}

\begin{rem} We recall that a Lie algebra is called quasi-Frobenius
if there exists a non-degenerate 2-cocycle on it. It is called Frobenius if the corresponding 2-cocycle is a coboundary.
Thus we see that the classification
problem for the Lagrangian subalgebras we are interested in contains the classification of Frobenius subalgebras of $\mathfrak{g}(\mathbb{K})$.
This question is quite complicated, as it is known from studying Frobenius subalgebras of $\mathfrak{g}$.
However, for $\mathfrak{g}=sl(2)$ there is only one Frobenius subalgebra, the standard parabolic one.

\end{rem}

\section{Lie bialgebra structures in Case II and Belavin-Drinfeld cohomologies}\label{case2}

Our task is to classify Lie bialgebra structures on $\mathfrak{g}(\mathbb{K})$
for which the associated classical double is isomorphic to $\mathfrak{g}(\mathbb{K})\oplus \mathfrak{g}(\mathbb{K})$.

\begin{lem}
Any Lie bialgebra structure $\delta$ on $\mathfrak{g}(\mathbb{K})$
for which the associated classical double is isomorphic to $\mathfrak{g}(\mathbb{K})\oplus \mathfrak{g}(\mathbb{K})$ is a coboundary $\delta=dr$ given by an $r$-matrix satisfying $r+r^{21}=f\Omega$, where $f\in\mathbb{K}$ and $\mathrm{CYB}(r)=0$.
\end{lem}

%\begin{rem}
%Since we have a field tower $\mathbb{C} \subset \mathbb{K} \subset \overline{\mathbb{K}}$ we can define %a natural action of the Galois group $\Gal(\overline{\mathbb{K}} / \mathbb{K} )$ on $\mathfrak{g}%(\overline{\mathbb{K}})$ and $G(\overline{\mathbb{K}})$.
%\end{rem}
Without loss of generality we may suppose that $f=1$. According to \cite{BD}, Lie bialgebra structures on a simple Lie algebra over
an algebraically closed field are coboundaries given by non-skewsymmetric $r$-matrices.
These $r$-matrices have been classified up to $\mathrm{Ad}(G)$-equivalence and they
are given in terms of admissible triples. (Recall that $G$ stands for a connected algebraic group with a reductive Lie algebra whose semisimple part is $\mathfrak{g}$.)

Let us fix a Cartan subalgebra $\mathfrak{h}$ of $\mathfrak{g}$ and the associated
root system. Fix a set of simple roots$\Gamma$. We choose a system of generators $e_{\alpha}$, $e_{-\alpha}$,
$h_{\alpha}$ such that $K(e_{\alpha},e_{-\alpha})=1$, for any positive root $\alpha$.
Denote by $\Omega_{0}$ the Cartan part of $\Omega$. Suppose also that $H \subset G$ is a maximal torus with Lie algebra $\mathfrak{h}$.

Let us recall from \cite {BD}, \cite{D} that any non-skewsymmetric $r$-matrix depends on certain discrete and continuous parameters.
The discrete one is an admissible triple
$(\Gamma_{1},\Gamma_{2},\tau)$, i.\ e.,
an isometry $\tau:\Gamma_{1}\longrightarrow \Gamma_{2}$ where $\Gamma_{1},\Gamma_{2}\subset\Gamma$ are such that
for any $\alpha\in\Gamma_{1}$ there exists $k\in \mathbb{N}$ satisfying
$\tau^{k}(\alpha)\notin \Gamma_{1}$.
The continuous parameter is a tensor $r_{0}\in \mathfrak{h}\otimes \mathfrak{h}$ satisfying $r_{0}+r_{0}^{21}=\Omega_{0}$
and $(\tau(\alpha)\otimes 1+1 \otimes \alpha)(r_{0})=0$ for any $\alpha\in \Gamma_{1}$. Then the associated $r$-matrix is given by the following formula
\[r_{BD}=r_{0}+\sum_{\alpha>0}e_{\alpha}\otimes e_{-\alpha}-\sum_{\alpha\in (\mathrm{Span} \Gamma_{1})^{+} }\sum_{k\in \mathbb{N}} e_{-\alpha}\wedge e_{\tau^{k}(\alpha)}.\]

Now, let us consider an $r$-matrix corresponding to a Lie bialgebra on $\mathfrak{g}(\mathbb{K})$.
Up to  $\mathrm{Ad}(G(\overline{\mathbb{K}}))$-equivalence, we have the Belavin--Drinfeld classification.
We may assume that our $r$-matrix is of the form $r_{X}=(\mathrm{Ad}_{X}\otimes \mathrm{Ad}_{X})(r_{BD})$, where $X\in G(\overline{\mathbb{K}})$  and $r_{BD}$ satisfies the system  $r+r^{21}=\Omega$ and $\mathrm{CYB}(r)=0$. The corresponding bialgebra structure is $\delta(a)=[r_{X},a\otimes1+1\otimes a]$ for any $a\in \mathfrak{g}(\mathbb{K})$.

Let us take an arbitrary $\sigma\in\Gal(\overline{\mathbb{K}}/\mathbb{K})$. Then we have $(\sigma\otimes \sigma)(\delta(a))=[\sigma(r_{X}),a\otimes1+1\otimes a]$ and $(\sigma\otimes \sigma)(\delta(a))=\delta(a)$, which imply that $\sigma(r_{X})=r_{X}+\lambda\Omega$, for some $\lambda\in\overline{\mathbb{K}}$. Let us show that $\lambda=0$. Indeed, $\Omega=\sigma(\Omega)=\sigma(r_{X})+\sigma(r_{X}^{21})= r_{X}+r_{X}^{21}+2\lambda\Omega$. Thus $\lambda=0$ and $\sigma(r_{X})=r_{X}$.
Consequently, we get
$(\mathrm{Ad}_{X^{-1}\sigma(X)}\otimes \mathrm{Ad}_{X^{-1}\sigma(X)})(\sigma(r_{BD}))=r_{BD}$.

\begin{defn}
Let $r$ be an $r$-matrix. The \emph{centralizer}
%$C(r,G_{\mathrm{ad}})$
$C(G, r)$
of $r$
is the set of all
$X\in G(\overline{\mathbb{K}})$
satisfying $(\mathrm{Ad}_{X}\otimes \mathrm{Ad}_{X})(r)=r$.
\end{defn}

\begin{thm}\label{CartanK}
For any simple Lie algebra $\mathfrak{g}$ and for any Belavin-Drinfeld matrix $r_{BD}$ we have
\[C(G,r_{BD}) \subset H,\]
where $H$ is a maximal torus of $G$.
\end{thm}

\begin{proof}

1) Let us consider the map
$\Phi:\mathfrak{g} \otimes \mathfrak{g} \to \mathfrak{g} \otimes \mathfrak{g}^*=\mathrm{End}(\mathfrak{g})$
induced by the natural pairing between $\mathfrak{g}$ and $\mathfrak{g}^*$ given by the Killing form, i.\ e.\
$$\Phi(a \otimes b)(u)=K(a,u)b.$$
Let $X \in C(G,r_{BD})$.
We have $$(\mathrm{Ad}_Xa \otimes \mathrm{Ad}_Xb)(u)=
K(\mathrm{Ad}_Xa,u)\mathrm{Ad}_Xb=\mathrm{Ad}_X(K(a, \mathrm{Ad}_Xu)b).$$
Thus, $X \in C(G,r_{BD})$ iff $\mathrm{Ad}_X\Phi(r)=\Phi(r)\mathrm{Ad}_X$.

2) $\mathrm{Ad}_X$ commutes with $\Phi(r)$ implies that it commutes with
semisimple and nilpotent parts of $\Phi(r)$. Our next aim is to compute them.
The operator $\Phi(e_\alpha \otimes e_\beta)$ maps $e_{-\alpha}$ to $e_\beta$ and the rest of the
Chevalley basis to zero. Hence, when $\alpha+\beta \not=0$ the operator
$\Phi(e_\alpha \otimes e_\beta)$ is nilpotent. Thus the operator $A=\Phi(\sum e_{\tau^k(\alpha)} \wedge e_{-\alpha})$ is nilpotent.

 For any positive root $\alpha$,
we have $\Phi(r_{DJ})e_\alpha=0$, $\Phi(r_{DJ})e_{-\alpha}=e_{-\alpha}$ and $\Phi(r_{DJ})h_{\pm \alpha}=\frac{1}{2}h_{\pm \alpha}$.
So when $\alpha$ and $\beta$ have opposite signs, $\Phi(r_{DJ})$ commutes with $\Phi(e_\alpha \otimes e_\beta)$.
Therefore, $\Phi(r_{DJ})$ commutes with $A$.
Clearly, $A(\mathfrak{h})=0$.
Hence,  both $A$ and $\Phi(r_{DJ})$ commute with $\Phi(s)$, where
$s=r-r_{DJ}-\sum e_{\tau^k(\alpha)} \wedge e_{-\alpha}\in \mathfrak{h}^{\otimes 2}$.

So we have the decomposition of $\Phi(r_{BD})$ in to the sum of
three commuting operators: $\Phi(r_{BD})=\Phi(r_{DJ})+\Phi(s)+A$. If $\Phi(s)=\Phi(s)_d+\Phi(s)_n$
is the Jordan decomposition of $\Phi(s)$ then $D=\Phi(r_{DJ})+\Phi(s)_d$ is semisimple, $N=A+\Phi(s)_n$
is nilpotent, and $D$ and $N$ commute. Thus, we have obtained the Jordan decomposition $\Phi(r_{BD})=D+N$.
Note that we have $De_{\alpha}=0$, $De_{-\alpha}=e_{-\alpha}$ and $Dh_\alpha \in \mathfrak{h}$.
It remains to show that the centralizer of $D$ lies in $H$.

3) The zero eigenspace $V_0$ of the operator $D$ contains
all positive root vectors and no negative root vectors. $\mathrm{Ad}_X$
commutes with $D$ and hence, must preserve $V_0$. But it also must preserve
its normalizer,  which is the Borel subalgebra $\mathfrak{b}^+$. Similarly, considering $V_1$
instead of $V_0$, we obtain that $\mathrm{Ad}_X$ preserves $\mathfrak{b}^-$.
Therefore, $\mathrm{Ad}_X$ preserves $\mathfrak{h}$. So, $X \in N_{G}(\mathfrak{h})$, the normalizer
of the Cartan subalgebra. Consequently, $\mathrm{Ad}_X $ induces an element of the Weyl group $W$.
It is well-known that $W$ acts transitively and without fixed points
on the set of the Borel subalgebras containing $\mathfrak{h}$.
But $\mathrm{Ad}_X$
preserves $\mathfrak{b}^+$. Therefore, $\mathrm{Ad}_X$ induces the unit of $W$ and thus, $X \in H$.
\end{proof}

For any root $\alpha$ we denote by $e^{\alpha}$ the corresponding character of the torus $H$.

\begin{thm}
If $(\Gamma_1, \Gamma_2, \tau)$ is an admissible triple corresponding to the
Belavin-Drinfeld matrix $r_{BD}$ then $X \in C(G,r_{BD})$ iff for any root
$\alpha \in \Gamma_1 \setminus \Gamma_2$ and for any $k \in \mathbb{N}$
we have $e^{\alpha}(X)=e^{\tau^k(\alpha)}(X)$ i.\ e., $e^{\alpha}(X)$ is constant on the strings of $\tau$.
\end{thm}

\begin{proof}
Vectors $e_\alpha \otimes e_{-\alpha}, h_\alpha \otimes h_\beta $ and
$e_\gamma \wedge e_\delta$ for $\gamma+\delta \neq 0$ form a set of
linearly independent eigenvectors of $\mathrm{Ad}_X$.
Hence, $X \in C(G,r_{BD})$ if and only if $\mathrm{Ad}_X$ preserves $e_{-\gamma} \wedge e_{\tau^k(\gamma)}$
for $\gamma \in \Gamma_1$. But this is equivalent to
$e^{\alpha}(X)=e^{\tau^k(\alpha)}(X)$ for any root $\alpha \in \Gamma_1 \setminus \Gamma_2$ and for any $k \in \mathbb{N}$.
\end{proof}

\begin{thm}\label{CartanK-2}
Let $r_{BD}$ be an $r$-matrix from the Belavin--Drinfeld list for
$\mathfrak{g}(\overline{\mathbb{K}})$. Suppose that \[(\mathrm{Ad}_{X^{-1}\sigma(X)}\otimes \mathrm{Ad}_{X^{-1}\sigma(X)})(\sigma(r_{BD}))=r_{BD}.\]
Then $\sigma(r_{BD})=r_{BD}$ and $X^{-1}\sigma(X)\in C(G,r_{BD})$.

\end{thm}

\begin{proof}
Consider $r=r_{BD}$ which corresponds to an admissible triple $(\Gamma_{1},\Gamma_{2},\tau)$ and $r_{0}\in \mathfrak{h}\otimes \mathfrak{h}$. Denote $Y:=X^{-1}\sigma(X)$ and $s:=r-r_{0}$. Then $(\mathrm{Ad}_Y\otimes \mathrm{Ad}_Y)(s+\sigma(r_{0}))=s+r_{0}$.

Following \cite{ES} p. 43--47, let $\Phi (r):\mathfrak{g}\longrightarrow \mathfrak{g}$ be defined as in Theorem \ref{CartanK}.
Let \[\mathfrak{g}_{r}^{\lambda}=
\bigcup_{n>0} \mathrm{Ker}(\Phi (r)-\lambda)^{n}.\] Then \[\mathfrak{g}=\mathfrak{g}_{r}^{0}\oplus \mathfrak{g}_{r}'\oplus \mathfrak{g}_{r}^{1}, \hspace{0.2cm}
\mathfrak{g}_{r}'=\bigoplus_{\lambda\neq 0,1} \mathfrak{g}_{r}^{\lambda}.\]

In our case,
$\mathfrak{n}_{-}\subseteq \mathfrak{g}_{s+r_{0}}^{0}\subseteq \mathfrak{b}_{-}$, $\mathfrak{n}_{+}\subseteq \mathfrak{g}_{s+r_{0}}^{1}\subseteq \mathfrak{b}_{+}$,
$\mathfrak{g}_{s+r_{0}}'\subseteq \mathfrak{h}$, $\mathfrak{g}_{s+r_{0}}^{0}+\mathfrak{g}_{s+r_{0}}'=\mathfrak{b}_{-}$ and $\mathfrak{g}_{s+r_{0}}^{1}+\mathfrak{g}_{s+r_{0}}'=\mathfrak{b}_{+}$. Similarly for $s+\sigma(r_{0})$.

On the other hand, it can be easily checked that
\[\Phi (\mathrm{Ad}_Y\otimes \mathrm{Ad}_Y)(r)=\mathrm{Ad}_Y\circ \Phi (r)\circ \mathrm{Ad}_Y^{-1}.\]
Hence, $\mathrm{Ad}_Y(\mathfrak{g}_{s+\sigma(r_{0})}^i)=
\mathfrak{g}_{s+r_{0}}^i$, $i=0,1$ and $\mathrm{Ad}_Y(\mathfrak{g}'_{s+\sigma(r_{0})})=\mathfrak{g}'_{s+r_{0}}$. Therefore,
$\mathrm{Ad}_Y(\mathfrak{b}_{\pm})=\mathfrak{b}_{\pm}$ and $\mathrm{Ad}_Y\in H (\overline{\mathbb{K}})$
since $G$ is connected. %(otherwise it is not true).

Let us analyse the equality $(\mathrm{Ad}_Y\otimes \mathrm{Ad}_Y)(s+\sigma(r_{0}))=s+r_{0}$. It follows that $(\mathrm{Ad}_Y\otimes \mathrm{Ad}_Y)(s)+\sigma(r_{0})=s+r_{0}$.
Taking into account that $r_{0},\sigma(r_{0})\in {\mathfrak h}^{\otimes 2}$
and \[(\mathrm{Ad}_Y\otimes \mathrm{Ad}_Y)(s)=\sum_{\alpha>0}e_{\alpha}\otimes e_{-\alpha}+\sum_{\beta\in (\mathbb{Z}\Gamma_{1})^{+}}\sum_{n>0} k_{\beta,n}e_{\beta}\wedge e_{-\tau^{n}(\beta)},\] for some integers $k_{\beta,n}$, we deduce that $\sigma(r_{0})=r_{0}$. Thus, $\sigma(r)=r$
and $\mathrm{Ad}_Y\in C(G,r)$.

\end{proof}

Henceforth we will assume that $r_{BD}$ is defined over $\mathbb{K}$, i.e. $r_0 \in \mathfrak{g}(\mathbb{K}) \otimes \mathfrak{g}(\mathbb{K})$.\\

In conclusion, $r_{X}=(\mathrm{Ad}_{X}\otimes \mathrm{Ad}_{X})(r_{BD})$
induces a Lie bialgebra structure on $\mathfrak{g}(\mathbb{K})$
if and only if $X\in G(\overline{\mathbb{K}})$
satisfies the condition $X^{-1}\sigma(X)\in C(G, r_{BD})$,
for any $\sigma\in \Gal(\overline{\mathbb{K}}/\mathbb{K})$.

\begin{defn}
Let $r_{BD}$ be a non-skewsymmetric $r$-matrix from the Belavin--Drinfeld list
and $C(G, r_{BD})$ its centralizer. We say that $X\in G(\overline{\mathbb{K}})$ is a \emph{Belavin--Drinfeld cocycle}
associated to $r_{BD}$ if $X^{-1}\sigma(X)\in C(G, r_{BD})$, for any $\sigma\in \Gal(\overline{\mathbb{K}}/\mathbb{K})$.

\end{defn}

We denote the set of Belavin--Drinfeld cocycles associated to $r_{BD}$ by
$Z(G, r_{BD})$. This set is non-empty, always contains the identity.

\begin{defn}
 Two cocycles $X_1$ and $X_{2}$ in $Z(G, r_{BD})$ are called \emph{equivalent} ($X_1\sim X_2$) if
there exists $Q\in G(\mathbb{K})$ and $C\in C(G, r_{BD})$ such that $X_{1}=QX_{2}C$.
\end{defn}

\begin{defn}
Let $H_{BD}^{1}(G, r_{BD})$ denote the set of equivalence classes of cocycles from $Z(G, r_{BD})$.
We call this set the \emph{Belavin--Drinfeld cohomology} associated to the $r$-matrix $r_{BD}$.
The Belavin--Drinfeld cohomology is said to be \emph{trivial} if all cocycles are equivalent to the identity, and
\emph{non-trivial} otherwise.
\end{defn}

We make the following remarks:
\begin{rem}
Assume that $X\in Z(G, r_{BD})$. Then for any $\sigma\in \Gal(\overline{\mathbb{K}}/\mathbb{K})$, $\sigma(X)=XC$,
for some $C\in C(G, r_{BD})$. We get $(\mathrm{Ad}_{\sigma(X)}\otimes \mathrm{Ad}_{\sigma(X)})(r_{BD})=(\mathrm{Ad}_{X}\otimes \mathrm{Ad}_{X})(r_{BD})$.

Consequently,
$(\mathrm{Ad}_{X}\otimes \mathrm{Ad}_{X})(r_{BD})$ induces a Lie bialgebra structure on $\mathfrak{g}(\mathbb{K})$.
\end{rem}

\begin{rem}
Assume that $X_1$ and $X_{2}$ in $Z(G, r_{BD})$ are equivalent. Then  $X_{1}=QX_{2}C$, for some  $Q\in G(\mathbb{K})$ and $C\in C(G, r_{BD})$.
This implies that
$(\mathrm{Ad}_{X_{1}}\otimes \mathrm{Ad}_{X_{1}})(r_{BD})=(\mathrm{Ad}_{QX_{2}}\otimes
\mathrm{Ad}_{QX_{2}})(r_{BD})$. In other words the $r$-matrices
$(\mathrm{Ad}_{X_{1}}\otimes \mathrm{Ad}_{X_{1}})(r_{BD})$
and $(\mathrm{Ad}_{X_{2}}\otimes \mathrm{Ad}_{X_{2}})(r_{BD})$ are gauge equivalent over
$\mathbb{K}$ via an element $Q\in G(\mathbb{K})$.
\end{rem}

The above remarks imply the following result.

\begin{prop}
Let $r_{BD}$ be a non-skewsymmetric $r$-matrix over $\overline{\mathbb{K}}$. There exists a one-to-one correspondence between $H_{BD}^{1}(G, r_{BD})$
and
gauge equivalence classes of Lie bialgebra structures on $\mathfrak{g}(\mathbb{K})$
with classical double $\mathfrak{g}(\mathbb{K})\oplus \mathfrak{g}(\mathbb{K})$ and $\overline{\mathbb{K}}$-isomorphic to $\delta(r_{BD})$.
\end{prop}

\section{Belavin-Drinfeld cohomologies for $sl(n)$}

Our next goal is to compute $H_{BD}^{1}(GL(n), r_{BD})$.
Let us first restrict ourselves to $\mathfrak{g}=sl(n)$
and the cohomology associated to the Drinfeld--Jimbo $r$-matrix $r_{DJ}$.
In this section we assume that $G=GL(n)$.

\begin{lem}\label{decomp}
Let $X\in GL(n,\overline{\mathbb{K}})$. Assume that for any $\sigma\in \Gal(\overline{\mathbb{K}}/\mathbb{K})$, $X^{-1}\sigma(X)\in \mathrm{diag}(n,\overline{\mathbb{K}})$. Then there exist $Q\in GL(n,\mathbb{K})$ and $D\in \mathrm{diag}(n,\overline{\mathbb{K}})$ such that $X=QD$.
\end{lem}

\begin{proof}
Let  $\sigma\in \Gal(\overline{\mathbb{K}}/\mathbb{K})$ and $\sigma(X)=XD_{\sigma}$, where $D_{\sigma}=\mathrm{diag}(d_{1},\ldots,d_{n})$. Here the elements $d_{i}$ depend on $\sigma$. Then $\sigma(x_{ij})=x_{ij}d_{j}$, for any $i$, $j$.

On the other hand, in each column of $X$ there exists a nonzero element. Let us denote these elements by $x_{i_{1}1}$, \ldots, $x_{i_{n}n}$. For $j=1$,  $\sigma(x_{i1})=x_{i1}d_{1}$ and
$\sigma(x_{i_{1}1})=x_{i_{1}1}d_{1}$. These relations imply that
$\sigma(x_{i1}/x_{i_{1}1})=x_{i1}/x_{i_{1}1}$ for any $\sigma\in \Gal(\overline{\mathbb{K}}/\mathbb{K})$ and thus $x_{i1}/x_{i_{1}1}\in\mathbb{K}$, for any $i$.

Similarly,
$x_{i2}/x_{i_{2}2}\in\mathbb{K}$, \ldots, $x_{in}/x_{i_{n}n}\in\mathbb{K}$, for any $i$. Let $Q=(k_{ij})$ be the matrix whose elements are
$k_{ij}=x_{ij}/x_{i_{j}j}$, for any $i$ and $j$.

Thus $X=QD$, where $Q\in GL(n,\mathbb{K})$ and
$D=\mathrm{diag}(x_{i_{1}1}, \ldots, x_{i_{n}n})$.

\end{proof}

\begin{prop}\label{case DJ}
For $\mathfrak{g}=sl(n)$, the Belavin--Drinfeld cohomology $H_{BD}^{1}(GL(n), r_{DJ})$ associated to $r_{DJ}$ and to the group
$GL(n)$ is trivial.
\end{prop}
%AAA
\begin{proof} It easily follows from the proof of Theorem \ref{CartanK} that the centralizer of $r_{DJ}$ is
$C(GL(n),r_{DJ})=\mathrm{diag}(n,\overline{\mathbb{K}})$.
Let us show that any cocycle is equivalent to the identity.
Indeed, let $X=(x_{ij})$ be a cocycle from $Z(GL(n), r_{DJ})$, i.\ e.,  $X^{-1}\sigma(X)\in C(GL(n),r_{DJ})$,
for any $\sigma\in \Gal(\overline{\mathbb{K}}/\mathbb{K})$.

It follows that $X^{-1}\sigma(X)\in\mathrm{diag}(n,\overline{\mathbb{K}}) $.
According to Lemma \ref{decomp}, there exists $Q\in GL(n,\mathbb{K})$ and
$D\in \mathrm{diag}(n,\overline{\mathbb{K}})$ such that $X=QD$. This proves that $X$ is equivalent to the identity.
\end{proof}

It turns out that the above result is true not only for $r_{DJ}$. Given an arbitrary
$r$-matrix $r_{BD}$ from the Belavin--Drinfeld list, the corresponding cohomology is also trivial.
First we will take a closer look to the centralizer $C(GL(n),r_{BD})$ of an $r$-matrix $r_{BD}$. Due to Theorem \ref{CartanK}, the following result holds.

\begin{lem}\label{lem1}
Let $r_{BD}$ be an arbitrary $r$-matrix from the Belavin--Drinfeld list. Then \[C(GL(n), r_{BD})\subseteq \mathrm{diag}(n,\overline{\mathbb{K}}).\]
\end{lem}

%\begin{rem}\label{cent_BD}
%The above result is not true for $o(2n)$ if we consider $O(2n)$ as the gauge group.
%However, one can easily show
%that in this case $C(r_{DJ}, O(2n))$ contains all diagonal matrices of $O(2n)$ (we will describe our presentation of $O(n)$ in
%Section 6).

%\end{rem}

For $sl(n)$ we are now able to give the exact description of $C(GL(n), r_{BD})$.

\begin{lem}\label{lem2}
$C(GL(n), r_{BD})$ consists of all diagonal matrices $T=\mathrm{diag}(t_{1},\ldots,t_{n})$ such that $t_{i}=s_{i}s_{i+1}\ldots s_{n}$,
where $s_{i}\in \overline{\mathbb{K}}$ satisfy the condition: $s_{i}=s_{j}$ if  $\alpha_{i}\in \Gamma_{1}$ and $\tau(\alpha_{i})=\alpha_{j}$.
\end{lem}

\begin{proof}
Let us assume that $r_{BD}$ is associated to an admissible triple $(\Gamma_{1},\Gamma_{2},\tau)$, where $\Gamma_{1}, \Gamma_{2}\subset\{\alpha_{1},\ldots,\alpha_{n-1}\}$. Let $T\in C(GL(n), r_{BD})$. According to Lemma \ref{lem1}, $T\in \mathrm{diag}(n,\overline{\mathbb{K}})$, therefore we put $T=\mathrm{diag}(t_{1},\ldots ,t_{n})$.
Now we note that $T\in C(GL(n), r_{BD})$ if and only if
$(\mathrm{Ad}_{T}\otimes \mathrm{Ad}_{T})(e_{\tau^{k}(\alpha)}\wedge e_{-\alpha})=
e_{\tau^{k}(\alpha)}\wedge e_{-\alpha}$ for any $\alpha\in \Gamma_{1}$ and any positive integer $k$.

For simplicity, let us take an arbitrary $\alpha_{i}\in \Gamma_{1}$ and suppose that $\tau(\alpha_{i})=\alpha_{j}$. We then get
$t_{i}t_{i+1}^{-1}=t_{j}t_{j+1}^{-1}$. Denote $s_{j}:=t_{j}t_{j+1}^{-1}$ for each $j\leq n-1$ and $s_{n}=t_{n}$. Then $t_{j}=s_{j}s_{j+1}\ldots s_{n}$ and moreover
$s_{i}=s_{j}$.

\end{proof}

\begin{thm}\label{sl}
For $\mathfrak{g}=sl(n)$, the Belavin--Drinfeld cohomology $H_{BD}^{1}(GL(n), r_{BD})$ associated to any $r_{BD}$ is trivial. Any Lie bialgebra structure on $\mathfrak{g}(\mathbb{K})$ is of the form $\delta(a)=[r,a\otimes 1+1\otimes a]$, where $r$ is an $r$-matrix which is $GL(n,\mathbb{K})$--equivalent to a non-skewsymmetric $r$-matrix from the Belavin--Drinfeld list.

\end{thm}

\begin{proof}
Let $X$ be a cocycle associated to $r_{BD}$ which is a fixed $r$-matrix from the Belavin--Drinfeld list. Thus $X^{-1}\sigma(X)$ belongs to the centralizer of the $r_{BD}$.
On the other hand, according to Lemma \ref{lem1}, $C(GL(n), r_{BD})\subseteq \mathrm{diag}(n,\overline{\mathbb{K}})$.

We then obtain that for any $\sigma\in \Gal(\overline{\mathbb{K}}/\mathbb{K})$,
$X^{-1}\sigma(X)$ is diagonal. By Lemma \ref{decomp}, we have a decomposition $X=QD$, where $Q\in GL(n,\mathbb{K})$ and $D\in \mathrm{diag}(n,\overline{\mathbb{K}})$. Since $\sigma(Q)=Q$, we have $X^{-1}\sigma(X)=(QD)^{-1}\sigma(QD)=D^{-1}Q^{-1}Q\sigma(D)=D^{-1}\sigma(D)$. Recall that $X^{-1}\sigma(X)\in C(GL(n),r_{BD})$. It follows that $D^{-1}\sigma(D)\in C(GL(n), r_{BD})$.

Let $D=\mathrm{diag}(d_{1},\ldots ,d_{n})$. Then $\mathrm{diag}(d_{1}^{-1}\sigma(d_{1}),\ldots ,d_{n}^{-1}\sigma(d_{n}))\in C(GL(n), r_{BD})$. Denote $t_{i}=d_{i}^{-1}\sigma(d_{i})$
and $T=\mathrm{diag}(t_{1},\ldots ,t_{n})$. According to Lemma \ref{lem2},
$T\in C(GL(n), r_{BD})$ if and only if $t_{i}t_{i+1}^{-1}=t_{j}t_{j+1}^{-1}$. Equivalently, $\sigma(d_{i}^{-1}d_{i+1}d_{j}d_{j+1}^{-1})=d_{i}^{-1}d_{i+1}d_{j}d_{j+1}^{-1}$. It follows that $d_{i}^{-1}d_{i+1}d_{j}d_{j+1}^{-1}\in \mathbb{K}$. Let $s_{i}:=d_{i}d_{i+1}^{-1}$
for any $i$ and $s_{n}=d_{n}$. Then we get $s_{j}s_{i}^{-1}\in \mathbb{K}$.

Let us fix a root $\alpha_{i_{0}}\in \Gamma_{1}\setminus\Gamma_{2}$ and let $\tau^{j}(\alpha_{i_{0}})=\alpha_{j}$. Then $s_{j}s_{i_{0}}^{-1}\in \mathbb{K}$, for any $j$. Denote $k_{j}:=s_{j}s_{i_{0}}^{-1}$.

On the other hand, $d_{j}=s_{j}s_{j+1}\ldots s_{n-1}s_{n}=k_{j}k_{j+1}\ldots k_{n}s_{i_{0}}^{n-j+1}$. Let \[K:=\mathrm{diag}(k_{1}k_{2}\ldots k_{n},k_{2}\ldots k_{n},\ldots , k_{n}),\]
\[C:=\mathrm{diag}(s_{i_{0}}^{n}, s_{i_{0}}^{n-1},\ldots , s_{i_{0}}).\] Note that
$D=KC$ and $K\in GL(n,\mathbb{K})$. Moreover, according to Lemma \ref{lem2}, $C\in C(GL(n), r_{BD})$.

Summing up, we have obtained that  if $X$ is any cocycle associated to $r_{BD}$, then $X=QD=QKC$, with $QK\in GL(n,\mathbb{K})$, $C\in C(GL(n), r_{BD})$.
This ends the proof.

\end{proof}

\section{Belavin-Drinfeld cohomologies for orthogonal algebras}

The next step in our investigation of Belavin--Drinfeld cohomologies
is for orthogonal algebras $o(m)$. We begin with
the case of Drinfeld--Jimbo $r$-matrix. In what follows, we will use the following split
form of the orthogonal algebra $o(n,\mathbb{C})$  and $o(n,\mathbb{K})$:

$$
o(n)=\{A\in gl(n): A^T S+SA=0\}
$$
Here $S$ is the matrix with 1 on the second diagonal and
zero elsewhere.
The group
$$
SO(n)=\{X\in SL(n): X^TSX=S \}
$$
naturally acts on $o(n)$.
It follows from Theorem \ref{CartanK} that $C(SO(n), r_{DJ})$ coincides with the maximal torus of $SO(n)$.
%The group $(O(n)$ has a center $\{\pm I\}$.
%The adjoint group $PO(n)=O(n)/\{\pm I\}$
%also acts naturally on $o(n)$. However, for our purposes it is more convenient to use $O(n)$ because it is a subgroup of $GL(n)$.
Our main result about Belavin-Drinfeld cohomologies for orthogonal algebras is the following:

\begin{thm} \label{orthog} Let $\mathfrak{g}=o(m)$ and $r_{DJ}$ be the Drinfeld--Jimbo $r$-matrix.
Then $H_{BD}^{1}(SO(m), r_{DJ})$ is trivial.

\end{thm}

\begin{proof}
(i) Assume $m=2n$. On $\overline{\mathbb{K}}^{m}$ let us fix the bilinear form
$$B(\mathbf{x},\mathbf{y})=\sum_{i=1}^{m} x_iy_{m+1-i}.$$

Let $X\in SO(m,\overline{\mathbb{K}})$ be a cocycle associated to $r_{DJ}$.
Thus $X^{-1}\sigma(X)\in C(SO(m), r_{DJ})$.
Recall that $C(SO(m), r_{DJ})=\mathrm{diag}(m,\overline{\mathbb{K}})\cap SO(m,\overline{\mathbb{K}})$.
Therefore $X^{-1}\sigma(X)\in \mathrm{diag}(m,\overline{\mathbb{K}})$. By Lemma
\ref{decomp}, one has the decomposition $X=QD$, where $Q\in GL(m,\mathbb{K})$ and $D\in
\mathrm{diag}(m,\overline{\mathbb{K}})$. Let us write $D=\mathrm{diag}(d_{1},\ldots ,d_{2n})$ and denote by $q_{i}$ the columns of $Q$.
Then $X=QD$ is equivalent to $Q^{T}SQ=D^{-1}SD^{-1}$, which in turn gives that $B(q_{i},q_{i'})d_{i}d_{i'}=\delta_{i}^{2n+1-i'}$.
We get $B(q_{i},q_{i'})=0$ if $i+i'\neq 2n+1$ and $B(q_{i},q_{2n+1-i})d_{i}d_{2n+1-i}=1$.
Let $k_{i}:=B(q_{i},q_{2n+1-i})$. Since  $Q\in GL(2n,\mathbb{K})$, $k_{i}\in \mathbb{K}$.
Because $k_{i}^{-1}= d_{i}d_{2n+1-i}$, it follows that $D=Q_{1}D_{1}$, where
$$Q_{1}=\mathrm{diag}(k_{1}^{-1},\ldots ,k_{n}^{-1},1\ldots ,1)\ \mathrm{and}$$
$$D_{1}=\mathrm{diag}(d_{1}k_{1},\ldots ,d_{n}k_{n},d_{n+1},\ldots ,d_{2n}).$$
We note that $X=(QQ_{1})D_{1}$, $D_1\in SO(2n)$ and hence, $D_{1}\in C(SO(2n), r_{DJ})$.
Then, clearly
$QQ_{1}\in SO(2n, \mathbb{K})$.
which proves that $X$ is equivalent to the identity.

(ii) Now consider $m=2n+1$.
By Lemma \ref{decomp}, we may write again $X=QD$, where $Q\in GL(m,\mathbb{K})$ and $D\in\mathrm{diag}(m,\overline{\mathbb{K}})$.

 Let $k_{i}:=B(q_{i},q_{2n+2-i})\in \mathbb{K}$. Repeating the computations as in (i), one gets $k_{i}^{-1}= d_{i}d_{2n+2-i}$.
If $i=n+1$, $d_{n+1}^{2}=k_{n+1}^{-1}\in\mathbb{K}$. This implies that either $d_{n+1}\in \mathbb{K}$ or $d_{n+1}\in j\mathbb{K}$, where $j^{2}=\hbar$.

Actually we can prove that the second case is impossible.

Let us denote $R=Q^{-1}$ and its rows by $r_{1}$,\ldots ,$r_{2n+1}$. Then $X^{T}SX=S$ is equivalent to $RSR^{T}=DSD$,
which in turn gives the following:
$B(r_{i},r_{i'})=0$, for all $i\neq i'$, $B(r_{i},r_{i})=d_{i}d_{2n+2-i}$ for all $i$.

Let us take an arbitrary orthogonal basis $v_{1}$,\ldots ,$v_{2n+1}$ in
$\mathbb{K}^{2n+1}$ and denote $B(v_{i},v_{i})=A_{i}$.

The matrix $V$ with rows $v_{i}$ satisfies $VSV^{T}=\mathrm{diag}(A_{1},\ldots ,A_{2n+1})$. This relation
implies that $A_{1}\ldots A_{2n+1}=(-1)^{n}\det(V)^{2}=((\sqrt{-1})^{n}\det(V))^{2}$. Therefore $A_{1}\ldots A_{2n+1}=l^{2}$ is a square of some $l\in \mathbb{K}$.

On the other hand, if $M$ is the change of basis matrix from $r_{i}$ to $v_{i}$,
then \[M^{T}\mathrm{diag}(A_{1},\ldots ,A_{2n+1})M=\mathrm{diag}(d_{1}d_{2n+1},\ldots ,d_{n+1}^{2},\ldots ,d_{2n+1}d_{1}).\]
By taking the determinant on both sides, one obtains
\[\det(M)^{2}A_{1}\ldots A_{2n+1}=
(d_{1}d_{2n+1})^{2}\ldots (d_{n}d_{n+2})^{2}d_{n+1}^{2}\]
which implies that $d_{n+1}^{2}$ is a square in $\mathbb{K}$, and consequently,
$d_{n+1}\in \mathbb{K}$.

Let us show that $X$ is equivalent to the trivial cocycle. Consider
 \[Q_{1}=\mathrm{diag}(k_{1}^{-1},\ldots ,k_{n}^{-1},d_{n+1},1,\ldots ,1),\]
\[D_{1}=\mathrm{diag}(d_{1}k_{1},\ldots ,d_{n}k_{n},1,d_{n+2}\ldots ,d_{2n+1}).\]
We have $D=Q_{1}D_{1}$ and $D_{1}\in SO(2n+1,\overline{\mathbb{K}})$. Thus $X=(QQ_{1})D_{1}$,
$QQ_{1}\in SO(2n+1, \mathbb{K})$, $D_{1}\in C(SO(2n+1), r_{DJ})$, i.\ e., $X$ is equivalent to the trivial cocycle, which completes the proof of triviality
of $H_{BD}^{1}(SO(m), r_{DJ})$.

\end{proof}

%\begin{rem}
%An embedding $G_1 \subset G$ gives us a natural surjective map $H_{BD}^1(G_1, r) \to H_{BD}^1(G,r)$ (for %both twisted and untwited cohomologies).  In particular, $H^1(O(m), r_{DJ})$ is trivial.
%\end{rem}
%DDD
%\vskip0.2cm
%We have just seen that the Belavin--Drinfeld cohomology $H_{BD}^{1}(SO(n), r_{DJ})$ is trivial
%for all $n$.

Regarding Belavin--Drinfeld cohomology $H_{BD}^{1}(SO(2n), r_{BD})$ for an arbitrary $r_{BD}$,
we can give an example where this set is non-trivial. Let us denote
the simple roots of $o(2n)$ by $\alpha_{i}=\epsilon_{i}-\epsilon_{i+1}$, for $i<n$, $\alpha_{n}=\epsilon_{n-1}+\epsilon_{n}$,
where $\{\epsilon_{i}\}$ is an orthonormal basis of $\mathfrak{h}^{*}$.
Let $\Gamma_{1}=\{\alpha_{n-1}\}$, $\Gamma_{2}=\{\alpha_{n}\}$ and
$\tau(\alpha_{n-1})=\alpha_{n}$.
Denote by $r_{BD}$ the $r$-matrix corresponding to the triple $(\Gamma_{1},\Gamma_{2},\tau)$ and $s$,
where $s\in \mathfrak{h}\wedge\mathfrak{h}$ satisfies $((\alpha_{n-1}-\alpha_{n}))\otimes 1)(2s)=((\alpha_{n-1}+\alpha_{n}))\otimes 1)\Omega_{0}$.

\begin{lem}
The centralizer $C(SO(2n), r_{BD})$ consists of all diagonal matrices of the form
\[T=\mathrm{diag}(t_{1},\ldots , t_{n-1},\pm 1, \pm 1, t_{n-1}^{-1},\ldots ,t_{1}^{-1})
,\] for arbitrary nonzero $t_{1},t_{2}\in \overline{\mathbb{K}}$.
\end{lem}

\begin{proof}
We already know that $C(SO(2n), r_{BD})\subseteq\mathrm{diag}(2n,\overline{\mathbb{K}})
\cap O(2n,\overline{\mathbb{K}})$. Let $T \in C(SO(2n), r_{BD})$, where $T=\mathrm{diag}(t_{1},\ldots ,t_{n}, t_{n}^{-1},\ldots ,t_{1}^{-1})$. Since $T$ commutes with
$r_{0}$ and $r_{DJ}$, $T \in C(SO(2n), r_{BD})$ if and only if $(\mathrm{Ad}_{T}\otimes \mathrm{Ad}_{T})(e_{\alpha_{n}}\wedge e_{\alpha_{n-1}})=e_{\alpha_{n}}\wedge e_{\alpha_{n-1}}$.
One can check that $(\mathrm{Ad}_{T}\otimes \mathrm{Ad}_{T})
(e_{\alpha_{n}}\wedge e_{\alpha_{n-1}})=t_{n}^{-2}e_{\alpha_{n}}\wedge e_{\alpha_{n-1}}$. Therefore we get
$t_{n}^{-2}=1$ and the conclusion follows.
\end{proof}

\begin{prop}
  Let $\mathfrak{g}=o(2n)$ and $r_{BD}$ be the $r$-matrix corresponding to the triple $(\Gamma_{1},\Gamma_{2},\tau)$,
and some $s\in \mathfrak{h}\wedge\mathfrak{h}$, where  $\Gamma_{1}=\{\alpha_{n-1}\}$, $\Gamma_{2}=\{\alpha_{n}\}$ and
$\tau(\alpha_{n-1})=\alpha_{n}$, and $((\alpha_{n-1}-\alpha_{n}))\otimes 1)(2s)=((\alpha_{n-1}+\alpha_{n}))\otimes 1)\Omega_{0}$. Then $H_{BD}^{1}(SO(2n), r_{BD})$ is non-trivial.
\end{prop}

\begin{proof}
Assume that $X^{-1}\sigma(X)\in C(SO(2n), r_{BD})$ for all
$\sigma\in \Gal(\overline{\mathbb{K}}/\mathbb{K})$.
By the above lemma, $X^{-1}\sigma(X)=\mathrm{diag}(t_{1},\ldots ,t_{n-1},\pm 1, \pm 1, t_{n-1}^{-1},\ldots ,t_{1}^{-1})$.

On the other hand, since $X^{-1}\sigma(X)$ is diagonal, it follows from Theorem
\ref{orthog} that there exist $Q\in SO(2n,\mathbb{K})$ and a diagonal matrix $D\in SO(2n,\overline{\mathbb{K}})$ such that $X=QD$.
Let us write $D=\mathrm{diag}(s_{1},\ldots ,s_{n}, s_{n}^{-1},\ldots ,s_{1}^{-1})$.
Since $Q\in O(2n,\mathbb{K})$, for any $\sigma\in \Gal(\overline{\mathbb{K}}/\mathbb{K})$, $\sigma(Q)=Q$. We obtain
$X^{-1}\sigma(X)=D^{-1}Q^{-1}Q\sigma(D)=D^{-1}\sigma(D)$, which
is equivalent to the following system:
$s_{i}^{-1}\sigma(s_{i})=t_{i}$, for all $i\leq n-1$ and $s_{n}^{-1}\sigma(s_{n})=\pm 1$.

Assume first that there exists $\sigma$ such that $\sigma(s_{n})=-s_{n}$. Then
$s_{n}\in j\mathbb{K}$. One can check that $X$ is equivalent to
$X_{0}=\mathrm{diag}(1,\ldots ,1,j,j^{-1},1,\ldots ,1)$ which is a non-trivial cocycle.

If  $\sigma(s_{n})=s_{n}$ for all $\sigma \in \Gal(\overline{\mathbb{K}}/\mathbb{K})$, then $s_{n}\in \mathbb{K}$. In this case,
\[D=\mathrm{diag}(s_{1},\ldots ,s_{n-1},1, 1,s_{n-1}^{-1},\ldots ,s_{1}^{-1})\cdot
\mathrm{diag}(1,\ldots ,1,s_{n}, s_{n}^{-1},1,\ldots ,1),\] where the first matrix is in $C(SO(2n),r_{BD})$ and the second
in $SO(2n,\mathbb{K})$. This proves that $X$ is equivalent to the identity cocycle.

\end{proof}

\section{Lie bialgebra structures in Case III and twisted Belavin-Drinfeld cohomologies}

Throughout this and the next sections we consider $GL(n)$ as the gauge group.
Here we analyse the Lie bialgebra structures on $\mathfrak{g}(\mathbb{K})$
for which the corresponding Drinfeld double is isomorphic to
$\mathfrak{g}(\mathbb{K}[j])$, where $j^2=\hbar$. The question is to find those subalgebras
$W$ of $\mathfrak{g}(\mathbb{K}[j])$ satisfying the
following conditions:

(i) $W\oplus \mathfrak{g}(\mathbb{K})=\mathfrak{g}(\mathbb{K}[j])$.

(ii) $W=W^{\perp}$ with respect to the non-degenerate symmetric bilinear form $Q$ given by $$Q(f_1+jf_2, g_1+jg_2)=K(f_1,g_2)+K(f_2,g_1).$$

We will restrict our discussion to $\mathfrak{g}=sl(n)$. We begin with the following remark. The field $\mathbb{K}[j]$ is endowed with a conjugation. For any element $a=f_1+jf_2$, its conjugate is $\overline{a}:=f_1-jf_2$. By the norm of an element $a \in \mathbb{K}[j]$ we will understand the element $a\overline{a} \in \mathbb{K}$.

If $A=A_{1}+jB_{1}$ and $B=A_{2}+jB_{2}$ are two matrices in $sl(n,\mathbb{K}[j])$, then $Q(A,B)=Tr(A_{1}B_{2}+B_{1}A_{2})$, i.\ e., the coefficient of $j$ in $Tr(AB)$.

\begin{lem}\label{emb}
Let $L$ be the subalgebra of $sl(n,\mathbb{K}[j])$ which consists of all matrices $Z=(z_{ij})$ satisfying $z_{ij}=\overline{z}_{n+1-i,n+1-j}$. Then  $L$ and $sl(n,\mathbb{K})$ are isomorphic via a conjugation of $sl(n,\mathbb{K}[j])$.
\end{lem}

\begin{proof}
Assume that $Z=(z_{ij})$ satisfies $z_{ij}=\overline{z}_{n+1-i,n+1-j}$. Then
$Z=S\overline{Z}S$, where $S$ is the matrix with 1 on the second diagonal and zero elsewhere.

Choose a matrix $X\in GL(n,\mathbb{K}[j])$ such that $\overline{X}=XS$. Then
$\overline{XZX^{-1}}=XS\overline{Z}SX^{-1}=XZX^{-1}$ which implies that $XZX^{-1}\in
sl(n,\mathbb{K})$. Conversely, if $A\in sl(n,\mathbb{K})$, then $Z=X^{-1}AX$
satisfies the condition $Z=S\overline{Z}S$.
\end{proof}

From now on we will identify $sl(n,\mathbb{K})$ with $L$. Let us find a complementary subalgebra to $L$ in $sl(n,\mathbb{K}[j])$. Let us denote by $\it{H}$ the Cartan subalgebra of $L$. If we identify
the Cartan subalgebra of $sl(n,\mathbb{K}[j])$ with $\mathbb{K}^{2(n-1)}$,
then $\it{H}$ is a Lagrangian subspace of $\mathbb{K}^{2(n-1)}$. Choose a Lagrangian subspace $H_{0}$ of $\mathbb{K}^{2(n-1)}$ such that $H_{0}$ has trivial intersection with $\it{H}$. Let $N^{+}$ be the algebra of upper triangular matrices of $sl(n,\mathbb{K}[j])$ with zero diagonal. Consider
$W_{0}=H_{0}\oplus N^{+}$. We immediately obtain the following

\begin{lem}
The subalgebra $W_{0}$ as above satisfies conditions (i) and (ii), where $sl(n,\mathbb{K})$ is identified with $L$ as in Lemma \ref{emb}.
\end{lem}

\begin{prop}
Any Lie bialgebra structure on $sl(n,\mathbb{K})$ for which the classical double is isomorphic to $sl(n,\mathbb{K}[j])$ is given by an $r$-matrix
which satisfies $CYB(r)=0$ and $r+r^{21}=j\Omega$.
\end{prop}

\begin{proof}
Let $W_{0}$ be as in the above lemma.
By choosing two dual bases in $W_{0}$ and $sl(n,\mathbb{K})$
respectively, one can construct the corresponding $r$-matrix $r_{0}$ over $\overline{\mathbb{K}}$. It is easily seen that $r_{0}$ satisfies the system $CYB(r_{0})=0$ and $r_{0}+r_{0}^{21}=j\Omega$.

Let us suppose that $W$ is another subalgebra of $sl(n,\mathbb{K}[j])$
satisfying conditions (i) and (ii). Then the corresponding $r$-matrix over $\overline{\mathbb{K}}$ is obtained by choosing dual bases in $W$ and $sl(n,\mathbb{K})$ respectively. We have $r+r^{21}=a\Omega$ for some $a\in\mathbb{K}[j]$.
On the other hand, the classical double of the Lie bialgebras corresponding to $r$ and $r_{0}$ is the same. This implies that $r$ and $r_{0}$ are
classical twists of each other and therefore $a=j$.
\end{proof}

On the other hand, over $\overline{\mathbb{K}}$, all $r$-matrices are gauge equivalent to the ones from Belavin--Drinfeld list. It follows that there exists a non-skewsymmetric $r$-matrix $r_{BD}$ and $X\in GL(n,\overline{\mathbb{K}})$ such that $r=j(\mathrm{Ad}_{X}\otimes \mathrm{Ad}_{X})(r_{BD})$.

Let us consider the conjugation on $\mathbb{K}[j]$ given by $\overline{a+bj}=a-bj$ and denote by  $\sigma_0$ its arbitrary lift  to $\Gal(\overline{\mathbb{K}}/\mathbb{K})$. We recall, see \cite{Ser}, that $\Gal(\overline{\mathbb{K}}/\mathbb{K})$ is generated by $\Gal(\overline{\mathbb{K}}/\mathbb{K}[j])$ and $\sigma_0$.

Consider an arbitrary $\sigma\in \Gal(\overline{\mathbb{K}}/\mathbb{K})$. Since $\delta$ is a cobracket on $sl(n,\mathbb{K})$, $(\sigma\otimes \sigma)(\delta(a))=\delta(a)$ and $(\sigma\otimes \sigma)(\delta(a))=[\sigma(r),a\otimes 1+1\otimes a]$.

Let us assume that $\sigma \in \Gal(\overline{\mathbb{K}}/\mathbb{K}[j])$. Exactly as in Section \ref{case2}, it follows that
$\sigma(r)=r$ and if $r=(\mathrm{Ad}_{X}\otimes \mathrm{Ad}_{X})(jr_{BD})$ with $X\in GL(n,\overline{\mathbb{K}})$, then $\sigma(X)=XD(\sigma)$.

 By the same arguments as in the proof of Lemma \ref{decomp}, the following result holds.

\begin{lem}\label{newdecomp}
Let $X\in GL(n,\overline{\mathbb{K}})$. Assume that for any $\sigma \in \Gal(\overline{\mathbb{K}}/\mathbb{K}[j])$, $X^{-1}\sigma(X)\in \mathrm{diag}(n,\overline{\mathbb{K}})$. Then there exists $P\in GL(n,\mathbb{K}[j])$ and
$D\in \mathrm{diag}(n,\overline{\mathbb{K}})$ such that $X=PD$.
\end{lem}

Now let us consider the action of $\sigma_{0}\in \Gal(\mathbb{K}[j]/\mathbb{K})$.
Our identities imply that $\sigma_{0}(r)=r+\alpha\Omega$, for some $\alpha\in\overline{\mathbb{K}}$. Let us show that $\alpha=-j$. Indeed, since $r+r^{21}=j\Omega$, we also have $\sigma_{0}(r)+\sigma_{0}(r^{21})=-j\Omega$. Combining these relations with $\sigma_{0}(r)=r+\alpha\Omega$, we get $\alpha=-j$ and therefore $\sigma_{0}(r)=r-j\Omega=-r_{21}$.

Recall now that $r=j(\mathrm{Ad}_{X}\otimes \mathrm{Ad}_{X})(r_{BD})$. It follows that $X\in GL(n,\overline{\mathbb{K}})$ must satisfy the identity
$(\mathrm{Ad}_{X^{-1}\sigma_{0}(X)}\otimes \mathrm{Ad}_{X^{-1}\sigma_{0}(X)})(\sigma(r_{BD}))=r_{BD}^{21}$. Using the same arguments as in the proof of Theorem
\ref{CartanK-2} in Section \ref{case2}, we obtain

\begin{prop}
Any Lie bialgebra structure on $sl(n,\mathbb{K})$ for which the classical double is $sl(n,\mathbb{K}[j])$ is given by an $r$-matrix $r=j(\mathrm{Ad}_{X}\otimes \mathrm{Ad}_{X})(r_{BD})$, where
$r_{BD}$ is a non-skewsymmetric $r$-matrix from the Belavin--Drinfeld list and $X\in GL(n,\overline{\mathbb{K}})$ satisfies \[(\mathrm{Ad}_{X^{-1}\sigma_{0}(X)}\otimes \mathrm{Ad}_{X^{-1}\sigma_{0}(X)})(r_{BD})=r_{BD}^{21}\] and, for $\sigma\in \Gal(\overline{\mathbb{K}}/\mathbb{K}[j])$,
  \[(\mathrm{Ad}_{X^{-1}\sigma(X)}\otimes \mathrm{Ad}_{X^{-1}\sigma(X)})(r_{BD})=r_{BD}.\]
\end{prop}

From now on we will assume that $r_{BD}$ is defined over $\mathbb{K}$ (i.e. its Cartan part $r_0$ is defined over $\mathbb{K}$).

\begin{defn}
Let $r_{BD}$ be a non-skewsymmetric $r$-matrix from the Belavin--Drinfeld list.
We call $X\in G(\overline{\mathbb{K}})$ a \emph{Belavin--Drinfeld twisted cocycle} associated to $r_{BD}$ if $(\mathrm{Ad}_{X^{-1}\sigma_{0}(X)}\otimes \mathrm{Ad}_{X^{-1}\sigma_{0}(X)})(r_{BD})=r_{BD}^{21}$ and for any $\sigma \in \Gal(\overline{\mathbb{K}}/\mathbb{K}[j])$, $(\mathrm{Ad}_{X^{-1}\sigma(X)}\otimes \mathrm{Ad}_{X^{-1}\sigma(X)})(r_{BD})=r_{BD}$.

\end{defn}

The set of Belavin--Drinfeld twisted cocycle associated to $r_{BD}$
will be denoted by $\overline{Z}(G, r_{BD})$.

Now, let us restrict ourselves to the case $r_{BD}=r_{DJ}$. In order to continue our investigation, let us prove the following
\begin{lem}\label{lem_S}
Let $S$ be the matrix with 1 on the second diagonal and zero elsewhere. Then
$$r_{DJ}^{21}=(\mathrm{Ad}_{S}\otimes \mathrm{Ad}_{S})r_{DJ}.$$
\end{lem}

\begin{proof}
We recall that $r_{DJ}$ is given by the following formula:
$$r_{DJ}=\sum_{\alpha>0}e_{\alpha}\otimes e_{-\alpha}+\frac{1}{2}\Omega_{0}$$ where $\Omega_{0}$ is the Cartan part of $\Omega$.

First note that $(\mathrm{Ad}_{S}\otimes \mathrm{Ad}_{S})(e_{ij}\otimes e_{ji})=e_{n+1-i,n+1-j}\otimes e_{n+1-j,n+1-i}$, which is a term in $r_{DJ}^{21}$, if $i>j$ (here $e_{ij}$ is a matrix with $1$ on the $(i,j)$ position and zero elsewhere).
On the other hand, since $\Omega_{0}$ is the Cartan part of the invariant element $\Omega$, we get $(\mathrm{Ad}_{S}\otimes \mathrm{Ad}_{S})\Omega_{0}=\Omega_{0}$.
This could also be proved by using the following: $\Omega_{0}=n\sum_{i=1}^{n}e_{ii}\otimes e_{ii}-I\otimes I$, where $I$ denotes the identity matrix of $GL(n,\mathbb{K})$. Then the identity $r_{DJ}^{21}=(\mathrm{Ad}_{S}\otimes \mathrm{Ad}_{S})r_{DJ}$ holds.

\end{proof}
\begin{defn}
By $J$ we denote the matrix with elements
$a_{kk}=1$, for $k\leq m$, $a_{kk}=-j$ for $k\geq m+1$,
$a_{k,n-k+1}=1$, for  $k\leq m$ and $a_{k,n-k+1}=j$ for $k\geq m+1$.
\end{defn}

\begin{lem}
$\overline{Z}(GL(n), r_{DJ})$ is non-empty.
\end{lem}
\begin{proof}
 Indeed, $\sigma_{0}(J)=JS$, $J \in GL(n, \mathbb{K}[j])$.
\end{proof}

\begin{cor}\label{corXP}
Let $X$ be a Belavin--Drinfeld twisted cocycle associated to $r_{DJ}$.
Then $X=PD$, where $P\in GL(n,\mathbb{K}[j])$ and $D\in \mathrm{diag}(n,\overline{\mathbb{K}})$. Moreover, $\sigma_{0}(P)=PSD_{1}$, where
$D_{1}\in \mathrm{diag}(n,\mathbb{K}[j])$.
\end{cor}
\begin{proof}
Since $X$ is a twisted cocycle, for any $\sigma \in \Gal(\overline{\mathbb{K}}/\mathbb{K}[j])$, $X^{-1}\sigma(X)\in C(GL(n),r_{DJ})$. Recall that $C(GL(n), r_{DJ})=\mathrm{diag}(n,\overline{\mathbb{K}})$. By Lemma \ref{newdecomp}, we have $X=PD$, where $P\in GL(n,\mathbb{K}[j])$ and $D\in \mathrm{diag}(n,\overline{\mathbb{K}})$. Lemma \ref{lem_S} implies that $S^{-1}X^{-1}\sigma_{0}(X)=:D_{2}\in \mathrm{diag}(n,\overline{\mathbb{K}})$. Since $X=PD$, $S^{-1}D^{-1}P^{-1}\sigma_{0}(P)\sigma_{0}(D)=D_{2}$. Hence $P^{-1}\sigma_{0}(P)=DSD_{0}\sigma_{0}(D^{-1})$.

Let $D_{1}:=S^{-1}DSD_{2}\sigma_{0}(D^{-1})\in \mathrm{diag}(n,\overline{\mathbb{K}})$. Then $\sigma_{0}(P)=PSD_{1}$ and $D_{1}\in \mathrm{diag}(n,\mathbb{K}[j])$.
\end{proof}

\begin{defn}
Let $X_{1}$ and $X_{2}$ be two Belavin--Drinfeld twisted cocycles associated to
 $r_{BD}$. We say that they are \emph{equivalent} if
there exist $Q\in GL(n,\mathbb{K})$ and $D\in \mathrm{diag}(n,\overline{\mathbb{K}})$ such that $X_{1}=QX_{2}D$.
\end{defn}

\begin{rem}\label{remXP}
Assume that $X$ is a twisted cocycle associated to $r_{DJ}$. By
Corollary \ref{corXP}, $X=PD$ and is equivalent to the twisted cocycle
$P\in GL(n,\mathbb{K}[j])$.
\end{rem}

\begin{defn}
Let $\overline{H}_{BD}^{1}(GL(n), r_{BD})$ denote the set of equivalence classes of twisted cocycles associated to $r_{BD}$. We call this set the \emph{Belavin--Drinfeld twisted cohomology} associated to the $r$-matrix $r_{BD}$.
\end{defn}
\begin{rem}
If $X_{1}$ and $X_{2}$ are equivalent, then the corresponding $r$-matrices
$r_{1}=j(\mathrm{Ad}_{X_{1}}\otimes \mathrm{Ad}_{X_{1}})(r_{DJ})$ and $r_{2}=j(\mathrm{Ad}_{X_{2}}\otimes \mathrm{Ad}_{X_{2}})(r_{DJ})$ are gauge equivalent via $Q\in GL(n,\mathbb{K})$.
\end{rem}

\begin{prop}
There is a one-to-one correspondence between
$\overline{H}_{BD}^{1}(GL(n), r_{BD})$ and
gauge equivalence classes of Lie bialgebra structures on $sl(n,\mathbb{K})$
with classical double $sl(n,\mathbb{K}[j])$
and $\overline{\mathbb{K}}$-isomorphic to $\delta(r_{BD})$.
\end{prop}

\begin{prop}\label{case DJ}
For $\mathfrak{g}=sl(n)$, the Belavin--Drinfeld twisted cohomology
$\overline{H}_{BD}^{1}(GL(n), r_{DJ})$ is non-empty and consists of one element.
\end{prop}

\begin{proof}
Let $X$ be a twisted cocycle associated to $r_{DJ}$. By Remark \ref{remXP},
$X$ is equivalent to a twisted cocycle $P\in GL(n,\mathbb{K}[j])$,
associated to $r_{DJ}$. We may therefore assume from the beginning that
$X\in GL(n,\mathbb{K}[j])$ and it remains to prove that all such cocycles
are equivalent.

Denote $m=n/2$ if $n$ is even or $m=(n+1)/2$ if $n$ is odd.
Denote by $J$ the matrix with elements
$a_{kk}=1$, for $k\leq m$, $a_{kk}=-j$ for $k\geq m+1$,
$a_{k,n-k+1}=1$, for  $k\leq m$ and $a_{k,n-k+1}=j$ for $k\geq m+1$. We will prove that $X$ and $J$ are equivalent,
i.\ e., $X=QJD'$, for some $Q\in GL(n,\mathbb{K})$ and $D'\in \mathrm{diag}(n,\mathbb{K}[j])$. The proof will be done by induction.

For $n=2$, set $J=\left(\begin{array}{cc} 1 & 1 \\
j &-j \end{array}\right)$ and let $X=\left(\begin{array}{cc} a & b \\
c & d \end{array}\right)\in GL(2,\mathbb{K}[j])$
satisfy $\overline{X}=XSD$ with $D=\mathrm{diag}(d_{1},d_{2})\in GL(2,\mathbb{K}[j])$. The identity is equivalent to the following system: $\overline{a}=bd_{1}$,  $\overline{b}=ad_{2}$, $\overline{c}=dd_{1}$,
$\overline{d}=cd_{2}$. Assume that $cd\neq 0$. Let $a/c=a'+b'j$. Then $b/d=a'-b'j$. One can immediately check that $X=QJD'$, where
$Q=\left(\begin{array}{cc} a' & b' \\
1 &0 \end{array}\right)\in GL(2,\mathbb{K})$, $D'=\mathrm{diag}(c,d)\in \mathrm{diag}(2,\mathbb{K}[j])$.

For $n=3$, consider $J=\left(\begin{array}{ccc} 1 & 0 &1\\
0&1&0\\j&0&-j\end{array}\right)$ and let $X=(a_{ij})\in GL(3,\mathbb{K}[j])$ satisfy $\overline{X}=XSD$, with
$D=\mathrm{diag}(d_{1},d_{2},d_{3})\in GL(3,\mathbb{K}[j])$. The identity is equivalent to the following system: $\overline{a_{11}}=d_{1}a_{13}$, $\overline{a_{21}}=d_{1}a_{23}$, $\overline{a_{31}}=d_{1}a_{33}$, $\overline{a_{12}}=d_{2}a_{12}$, $\overline{a_{22}}=d_{2}a_{22}$, $\overline{a_{32}}=d_{2}a_{32}$, $\overline{a_{13}}=d_{3}a_{11}$, $\overline{a_{23}}=d_{3}a_{21}$, $\overline{a_{33}}=d_{3}a_{31}$. Assume that $a_{21}a_{22}a_{23}\neq 0$.

Let  $a_{11}/a_{21}=b_{11}+b_{13}j$ and $a_{31}/a_{21}=b_{31}+b_{33}j$. Then $a_{13}/a_{23}=b_{11}-b_{13}j$ and $a_{33}/a_{23}=b_{31}-b_{33}j$. On the other hand, let $b_{12}:=a_{12}/a_{22}$ and $b_{32}:=a_{32}/a_{22}$. Note that $b_{12}\in \mathbb{K}$, $b_{32}\in \mathbb{K}$. One can immediately check that $X=QJD'$, where
$Q=\left(\begin{array}{ccc} b_{11} & b_{12} &b_{13}\\
1&1&0\\ b_{31} & b_{32} & b_{33}\end{array}\right)\in GL(3,\mathbb{K})$,
$D'=\mathrm{diag}(a_{21},a_{22},a_{23})\in \mathrm{diag}(3,\mathbb{K}[j])$.

For $n>3$, we proceed by induction. Let us denote the constructed above $J\in GL(n,\mathbb{K}[j])$ by $J_n$. We are going to prove that if $X\in GL(n,\mathbb{K}[j])$ satisfies $\overline{X}=XSD$, then using elementary row operations with entries from $\mathbb{K}$ and multiplying columns by proper elements from
$\mathbb{K}[j]$ we can bring $X$ to $J_n$.

We will need the following operations on a matrix
$$M=(m_{pq})\in \mathrm{Mat}(n):
$$

1. $u_n (M)=(m_{pq})\in \mathrm{Mat}(n-2),\ p,q=2,3,\ldots ,n-1$;

2. $g_n (M)=(m_{pq})\in \mathrm{Mat}(n+2),$ where $m_{pq}$ are already defined for $p,q=1,2,\ldots n$,
$m_{00} =m_{n+1,n+1}=1$ and the rest $m_{0,a}=m_{a,0}=m_{n+1,a}=m_{a,n+1}=0$.
\vskip0.2cm

It is clear that $u_n (X)$ satisfies the twisted cocycle condition. However, its determinant might vanish. To avoid this complication, we note that columns $2,3,\ldots ,n-1$ of $X$ are linearly independent. Applying
elementary row operations (in fact, they are permutations) we obtain a new cocycle $X_1$, which is
equivalent to $X$ and such that $u_n (X_1)$
is a cocycle in $GL(n-2, \mathbb{K}[j])$.
Then, by induction, there exist $Q_{n-2}\in GL(n-2, \mathbb{K})$ and a diagonal matrix $D_{n-2}$ such that
$$
Q_{n-2}\cdot u_n (X_1)\cdot D_{n-2} =J_{n-2}
$$
Let us consider $X_n=g_{n-2} (Q_{n-2} )\cdot X_1\cdot g_{n-2} (D_{n-2})$. Clearly, $X_n$ is a twisted cocycle
equivalent to $X$ and $u_n (X_n )=J_{n-2}$.

Applying elementary row operations with entries from $\mathbb{K}$ and multiplying by a proper
diagonal matrix, we can obtain a new cocycle
$Y_n=(y_{pg})$ equivalent to $X$ with the following properties:

1. $u_n (Y_n)=J_{n-2}$;

2. $y_{12}=y_{13}=\ldots =y_{1,n-1}=0$ and $y_{n2}=y_{n3}=\ldots =y_{n,n-1}=0$;

3. $y_{11}=y_{1n}=1$, here we use the fact that if $y_{pq}=0$, then $y_{p,n+1-q}=0$.

It follows from the cocycle condition $\overline{Y_n} =Y_n \cdot S\cdot\mathrm{diag}(h_1,\ldots ,h_n)$
that $h_1=h_n=1$ and hence, $y_{n1}=\overline{y_{nn}}$.

Now, we can use the first row to achieve  $y_{n1}=-y_{nn}=j$ and after that, we use the first
and the last rows to get $y_{k1}=0$, $k=2,\ldots ,n-1$. Then  the elements $y_{kn}$, $k=2,\ldots ,n-1$ will vanish automatically.  We have obtained $J_n$ from $X$ and thus, have proved that $X$ is equivalent to $J_n$.

\end{proof}

\begin{ex}

For $\mathfrak{g}=sl(2)$, the Belavin--Drinfeld list of non-skewsymmetric constant $r$-matrices consists of only one class, $r_{DJ}=e\otimes f+\frac{1}{4}h\otimes h$, where $e=e_{12}$, $f=e_{21}$
and $h=e_{11}-e_{22}$. One can easily determine the
corresponding class of gauge equivalent Lie bialgebra structures on $sl(2,\mathbb{K})$
with classical double $sl(2,\mathbb{K}[j])$
and $\overline{\mathbb{K}}$-isomorphic to $\delta(r_{DJ})$.
Indeed, we have seen that the corresponding Lie bialgebra structure
$\delta=dr$, where the $r$-matrix is $r=j(\mathrm{Ad}_{X}\otimes \mathrm{Ad}_{X})r_{DJ}$ and $X$ is a twisted cocycle. On the other hand, according to the
above result, any such $X$ is equivalent to \[J=\left(\begin{array}{cc} 1 & 1 \\
j &-j \end{array}\right).\] Therefore a class representative
is $\delta_{0}=dr_{0}$, where $r_{0}=j(\mathrm{Ad}_{J}\otimes \mathrm{Ad}_{J})r_{DJ}$.
A straightforward computation gives $$r_{0}=\frac{j\Omega}{2}+\frac{1}{4}h\wedge e+\frac{\hbar}{4}f\wedge h. $$
We conclude that any Lie bialgebra structure on $sl(2,\mathbb{K})$
with classical double $sl(2,\mathbb{K}[j])$ is gauge equivalent
to that given by $a\cdot dr_0 ,\ a\in \mathbb{K}$.

\end{ex}

\begin{rem}
In case $sl(2)$, it follows that the Drinfeld--Jimbo $r$-matrix multiplied by $a\in \mathbb{K}$
along with $ar_0$,  $r_{0}=\frac{j\Omega}{2}+\frac{1}{4}h\wedge e+\frac{\hbar}{4}f\wedge h$,
provides all $GL(n)$  non-equivalent Lie bialgebra structures on $sl(2,\mathbb{K})$ of types II and III
and, consequently, two families of non-isomorphic Hopf algebra structures on $U(sl(2,\mathbb{C}))[[\hbar]]$.
Moreover, in some sense
these two structures exhaust all Hopf algebra
structures on  $U(sl(2,\mathbb{C}))[[\hbar]]$ with a non-trivial Drinfeld associator (see also conjectures below).
\end{rem}

\begin{rem} The next step would be to compute the Belavin--Drinfeld twisted cohomology corresponding
to an arbitrary $r$-matrix $r_{BD}$. Unlike untwisted cohomology, it might happen that even $\overline{Z}(G, r_{BD})$
is empty as we will see in next section.

\end{rem}

\section{Twisted cohomologies for $sl(n)$ of Cremmer-Gervais type}
In this section the gauge group $G$  is always $GL(n)$.
We have seen that $\overline{H}^{1}_{BD}(GL(n),r_{DJ})$,
where $r_{DJ}$ is the Drinfeld--Jimbo $r$-matrix, consists of one element. We will now turn our attention to other non-skewsymmetric $r$-matrices
and analyse the corresponding twisted cohomology set. Let us consider an arbitrary admissible triple $(\Gamma_{1},\Gamma_{2},\tau)$, and a tensor $r_{0}\in \mathfrak{h}\otimes \mathfrak{h}$ satisfying $r_{0}+r_{0}^{21}=\Omega_{0}$
and $(\tau(\alpha)\otimes 1+1 \otimes \alpha)(r_{0})=0$ for any $\alpha\in \Gamma_{1}$. We recall that the associated $r$-matrix is
given by the following formula
\[r=r_{0}+\sum_{\alpha>0}e_{\alpha}\otimes e_{-\alpha}+\sum_{\alpha\in (\mathrm{Span} \Gamma_{1})^{+} }\sum_{k\in \mathbb{N}} e_{\alpha}\wedge e_{-\tau^{k}(\alpha)}.\]

Assume now that there exists $X\in \overline{Z}(GL(n),r)$. Then $r$ and $r^{21}$ are gauge equivalent since $(\mathrm{Ad}_{X^{-1}\sigma_{2}(X)}\otimes \mathrm{Ad}_{X^{-1}\sigma_{2}(X)})(r)=r^{21}$.

Let $S\in GL(n,\mathbb{K})$ be the matrix with 1 on the second diagonal and 0 elsewhere. Let us denote by $s$ the automorphism of the Dynkin diagram given by $s(\alpha_{i})=\alpha_{n-i}$
for all $i=1,\ldots ,n-1$. Clearly, $\mathrm{Ad}_{S}(e_{\alpha})=e_{-s(\alpha)}$
and $\mathrm{Ad}_{S}(e_{-\tau^{k}(\alpha)})= e_{s\tau^{k}(\alpha)}$. Thus

$$(\mathrm{Ad}_{S}\otimes \mathrm{Ad}_{S})(r)=(\mathrm{Ad}_{S}\otimes \mathrm{Ad}_{S})(r_{0})+\sum_{\alpha>0}e_{-s(\alpha)}\otimes e_{s(\alpha)}+$$
$$\sum_{\alpha\in (\mathrm{Span} \Gamma_{1})^{+} }\sum_{k\in \mathbb{N}} e_{-s(\alpha)}\wedge e_{s\tau^{k}(\alpha)}.$$

On the other hand, since $r$ and $r^{21}$ are gauge equivalent, $(\mathrm{Ad}_{S}\otimes \mathrm{Ad}_{S})(r)$ and $r^{21}$ must be gauge equivalent as well. The following condition has to be fulfilled for all $k$: $s(\alpha)=\tau^{k}(\beta)$ if $\beta=s\tau^{k}(\alpha)$. We get $s\tau=\tau^{-1}s$,
$s(\Gamma_{1})=\Gamma_{2}$ (and $s(\Gamma_{2})=\Gamma_{1}$). In conclusion we have obtained

\begin{prop}\label{cond_tau}
Let $r$ be a non-skewsymmetric $r$-matrix associated to an admissible triple
$(\Gamma_{1},\Gamma_{2},\tau)$. If $\overline{Z}(GL(n),r)$ is non-empty,
then $s(\Gamma_{1})=\Gamma_{2}$ and $s\tau=\tau^{-1}s$.

\end{prop}

The following two results will prove to be quite useful for the
investigation of the twisted cohomologies for arbitrary non-skewsymmetric $r$-matrices.

\begin{lem}\label{equiv}

Assume $X\in \overline{Z}(GL(n),r)$. Then there exists a twisted cocycle $Y\in GL(n,\mathbb{K}[j])$, associated to $r$, and equivalent to $X$.

\end{lem}

\begin{proof}
We have $X\in GL(n,\overline{\mathbb{K}})$ and for any $\sigma\in \Gal(\overline{\mathbb{K}}/\mathbb{K}[j])$,
$X^{-1}\sigma(X)\in C(GL(n), r)$. On the other hand, the Belavin--Drinfeld cohomology for $sl(n)$ associated to $r$ is trivial. This implies that $X$ is equivalent to the identity, where in the equivalence relation we consider $\mathbb{K}[j]$
instead of $\mathbb{K}$. So there exists $Y\in GL(n,\mathbb{K}[j])$ and $C\in C(GL(n), r)$ such that $X=YC$.
 Since $(\mathrm{Ad}_{X^{-1}\sigma_{0}(X)}\otimes \mathrm{Ad}_{X^{-1}\sigma_{0}(X)})(r)=r^{21}$, $(\mathrm{Ad}_{Y^{-1}\sigma_{0}(Y)}\otimes \mathrm{Ad}_{Y^{-1}\sigma_{0}(Y)})(r)=r^{21}$. Thus $Y$ is also a twisted cocycle associated to $r$.
\end{proof}

Recall that $J\in GL(n,\mathbb{K}[j])$ denotes the matrix
with entries
$a_{kk}=1$ for $k \leq m$, $a_{kk}=-j$ for $k \geq m+1$, $a_{k,n+1-k}=1$
for $k \leq m$, $a_{k,n+1-k}=j$ for $k \geq m+1$, where $m=[\frac{n+1}{2}]$.

\begin{lem}\label{decRJD}
Let $r$ be a non-skewsymmetric $r$-matrix
associated to an admissible triple $(\Gamma_{1},\Gamma_{2},\tau)$
satisfying $s(\Gamma_{1})=\Gamma_{2}$ and $s\tau=\tau^{-1}s$. If
$X\in \overline{Z}(GL(n), r)$, then there exist $R\in GL(n,\mathbb{K})$ and $D\in \mathrm{diag}(n, \overline{\mathbb{K}})$ such that $X=RJD$.

\end{lem}

\begin{proof}
According to Lemma \ref{equiv}, $X=YC$, where $Y\in GL(n, \mathbb{K}[j])$
and $C\in C(GL(n), r)$. Since $(\mathrm{Ad}_{Y^{-1}\sigma_{0}(Y)}\otimes \mathrm{Ad}_{Y^{-1}\sigma_{0}(Y)})(r)=r^{21}$ and $(\mathrm{Ad}_{S}\otimes \mathrm{Ad}_{S})(r)=r^{21}$,
it follows that $S^{-1}Y^{-1}\sigma_{0}(Y)\in C(GL(n), r)$. On the other hand,
by Lemma \ref{lem1}, $C(GL(n), r)\subset \mathrm{diag}(n,\overline{\mathbb{K}})$. We get
$S^{-1}Y^{-1}\sigma_{0}(Y)\in \mathrm{diag}(n,\overline{\mathbb{K}})$. Now Theorem \ref{case DJ} implies that $Y=RJD_{0}$, where $R\in GL(n,\mathbb{K})$ and $D_{0}\in \mathrm{diag}(n,\overline{\mathbb{K}})$. Consequently, $X=RJD_{0}C=RJD$ with
$D=D_{0}C\in \mathrm{diag}(n,\overline{\mathbb{K}})$.

\end{proof}

We will now look for admissible triples which satisfy condition $s\tau=\tau^{-1}s$. Let us consider the
Cremmer--Gervais triple: $\Gamma_{1}=\{\alpha_{1}, \alpha_{2},\ldots ,\alpha_{n-2}\}$,
$\Gamma_{2}=\{\alpha_{2}, \alpha_{3},\ldots ,\alpha_{n-1}\}$ and $\tau(\alpha_{i})=
\alpha_{i+1}$. Clearly, $s\tau=\tau^{-1}s$. Denote by $r_{CG}$ the
Cremmer--Gervais $r$-matrix corresponding to the above triple and whose Cartan part is given by the following expression:
\[r_{0}=\frac{1}{2}\sum_{i=1}^{n}e_{ii}\otimes e_{ii}+\sum_{1\leq i<k\leq n}
\frac{n+2(i-k)}{2n}e_{ii}\otimes e_{kk}.\]

We intend to describe $\overline{H}^{1}_{BD}(GL(n),r_{CG})$. Let us first analyse the case $\mathfrak{g}=sl(3)$. The centralizer $C(GL(n), r_{CG})$ consists of diagonal matrices $\mathrm{diag}(a,b,c)$
such that $b^{2}=ac$. Consider
\[J=\left(\begin{array}{ccc} 1 & 0 & 1\\
0&1&0\\j & 0 & -j\end{array}\right).\]
\begin{lem}
Let $X\in GL(3,\mathbb{K}[j])$. Then $\overline{X}=XSC$, where $C\in C(GL(n), r_{CG})$
if and only if $X=RJ\mathrm{diag}(p,q,r)$, with $R\in GL(3,\mathbb{K})$ and
$prq^{-2}=k\in \mathbb{K}$.
\end{lem}

\begin{proof}
According to Lemma \ref{decRJD},
there exist $R\in GL(3,\mathbb{K})$ and $D=\mathrm{diag}(p,q,r)$,
$p,q,r\in \mathbb{K}[j]$ such that
$X=RJD$.  We get $ \overline{X}=RJS\overline{D}=
RJDD^{-1}S\overline{D}=XS\mathrm{diag}(\overline{p}r^{-1},\overline{q}q^{-1},
\overline{r}p^{-1})$. Let $C=\mathrm{diag}(\overline{p}r^{-1},\overline{q}q^{-1},
\overline{r}p^{-1})$. Then  $C\in C(GL(n), r_{CG})$ if and only if
$\overline{p}\overline{r}(pr)^{-1}=(\overline{q}q^{-1})^{2}$, which is
equivalent to $\overline{prq^{-2}}=prq^{-2}$, i.\ e., $prq^{-2} \in \mathbb{K}$.

\end{proof}

\begin{prop}
$\overline{H}^{1}_{BD}(GL(3),r_{CG})$ consists of one element, namely $J$
can be chosen as a representative.
\end{prop}

\begin{proof}
Let $X\in \overline{Z}(GL(3),r_{CG})$. According to the preceding lemma,
$X=RJ\mathrm{diag}(p,q,r)$, with $R\in GL(3,\mathbb{K})$ and
$prq^{-2}=k\in \mathbb{K}$. We distinguish the following cases:

\textit{Case 1.} Let $k=l^{-2}$, where $l\in \mathbb{K}$. Then we have a
particular solution to the equation $prq^{-2}=l^{-2}$, namely
$p_{0}=r_{0}=1$, $q_{0}=l$. By setting $p=p_{0}p_{1}$, $q=q_{0}q_{1}$,
$r=r_{0}r_{1}$, we see that $\mathrm{diag}(p_{1},q_{1},r_{1})\in C(GL(n), r_{CG})$
and $\mathrm{diag}(p_{0},q_{0},r_{0})=\mathrm{diag}(1,l,1)$
which commutes with $J$.
It follows that $X=RJ\mathrm{diag}(1,l,1)\cdot \mathrm{diag}(p_{1},q_{1},r_{1})$,
equivalently, $X=R_{1}J \mathrm{diag}(p_{1},q_{1},r_{1})$,
where $R_{1}:=R\cdot\mathrm{diag}(1,l,1)$. Consequently, $X$ is equivalent to
$J$.

\textit{Case 2.} Suppose $k$ is not a square of an element of $\mathbb{K}$. In this case, without loss of generality, we can set $l=j$ and $k=\hbar$. We want to prove that $J\cdot \mathrm{diag}(1,j,1)=R'JC'$,
for some $R'\in GL(3,\mathbb{K})$
and some $C'=\mathrm{diag}(x,y,z)$ with $xy^{-2}z=1$.
Equivalently, $J\cdot \mathrm{diag}(x^{-1},jy^{-1},z^{-1})J^{-1}=R'$.
Since $\overline{R'}=R'$, we get $\overline{J}\mathrm{diag}(\overline{x}^{-1},-j\overline{y}^{-1},
\overline{z}^{-1})\overline{J}^{-1}=J\mathrm{diag}(x^{-1},jy^{-1},z^{-1})J^{-1}$.
Thus $\mathrm{diag}(\overline{x}^{-1},-j\overline{y}^{-1},
\overline{z}^{-1})=\mathrm{diag}(x^{-1},jy^{-1},z^{-1})$. We obtained that
$x=\overline{z}$ and $y=kj$, with $k\in \mathbb{K}$. Hence, we have to find
$x$ and $k$ so that $x\overline{x}=k^{2}\hbar$. Clearly, it is sufficient
to find $\alpha\in \mathbb{K}[j]$ with norm $\hbar$ (recall that the norm of an element $a \in \mathbb{K}[j]$ is the element $a\overline{a} \in \mathbb{K}$). The latter is trivial
because we can for instance choose $\alpha=ij$ ($i^{2}=-1$).
Thus the existence of $R'\in GL(3,\mathbb{K})$ and $C'=\mathrm{diag}(x,y,z)$
is proved and therefore we conclude that $X$ is equivalent to
$J$.

\end{proof}

The above result can be generalized to $sl(n)$, $n>3$. Let us first note that
the centralizer $C(GL(n), r_{CG})$ consists of diagonal matrices
$\mathrm{diag}(p_{1},p_{2},\ldots ,p_{n})$ such that $p_{i+1}=p_{2}^{i}p_{1}^{1-i}$
for all $i$. Let $m=[\frac{n+1}{2}]$.

\begin{lem}
Let $X\in GL(n,\mathbb{K}[j])$. Then $\overline{X}=XSC$, where $C\in C(GL(n), r_{CG})$
if and only if $X=RJ\mathrm{diag}(d_{1},\ldots ,d_{n})$, with $R\in GL(n,\mathbb{K})$, $d_{1},\ldots ,d_{m}\in \mathbb{K}[j]$ and $d_{n-i+1}=\overline{d_{i}}r^{i-2}q^{-1}$, for $i\leq m$, where $r,q$ are such that $r^{n-3}=q\overline{q}$.

\end{lem}

\begin{proof}
According to Lemma \ref{decRJD}, there exist $R\in GL(n,\mathbb{K})$, $D=\mathrm{diag}(d_{1},\ldots ,d_{n})$,
$d_{i}\in \mathbb{K}[j]$ such that
$X=RJD$.  We get $\overline{X}=RJS\overline{D}=
RJDD^{-1}S\overline{D}=XS(SD^{-1}S\overline{D})$. On the other hand,
$SD^{-1}S\overline{D}=\mathrm{diag}(\overline{d_{1}}d_{n}^{-1},\overline{d_{2}}d_{n-1}^{-1},\ldots , \overline{d_{n}}d_{1}^{-1})$. Denote $p_{i}=\overline{d_{i}}d_{n+1-i}^{-1}$. Obviously, $p_{n+1-i}=(\overline{p_{i}})^{-1}$. But
$\mathrm{diag}(p_{1},p_{2},\ldots ,p_{n})$ belongs to $C(GL(n), r_{CG})$ if and only if $p_{i+1}=p_{2}^{i}p_{1}^{1-i}$ for all $i$.
It follows that $p_{2}^{n-i}p_{1}^{1+i-n}=(\overline{p_{2}})^{-i+1}
(\overline{p_{1}})^{i-2}$ must be fulfilled for all $i$. For $i=1$ we get $p_{2}^{n-1}=p_{1}^{n-1}\overline{p_{1}}^{-1}$ (note that if this identity holds then the other identities are true for all $i$).  The identity is also equivalent to the following: $p_{1}^{n-3}=p_{2}^{n-2}\overline{p_{2}}$. Set $p_{1}=qr$,
$p_{2}=q$. Then $r^{n-3}=q\overline{q}$. We obtain $d_{n-i+1}=\overline{d_{i}}r^{i-2}q^{-1}$, for all $i\leq m$. Let us note that if $n=2m-1$, we have $d_{m}(\overline{d_{m}})^{-1}=r^{m-2}q^{-1}$. Since the norm of $r^{m-2}q^{-1}$ is 1, this condition is consistent.
\end{proof}

\begin{rem}
It follows from the above lemma that $X=RJ$, where $R\in GL(n,\mathbb{K})$, is a twisted cocycle associated to $r_{CG}$. All such cocycles are equivalent to
$J$.
\end{rem}

\begin{prop}{\label{CG}}
$\overline{H}^{1}_{BD}(GL(n),r_{CG})$ consists of one element, namely $J$
can be chosen as a representative.
\end{prop}

\begin{proof}
Let $X\in \overline{Z}(GL(n), r_{CG})$. According to the previous lemma,
$X=RJ\mathrm{diag}(d_{1},\ldots ,d_{n})$,
where $d_{n-i+1}=\overline{d_{i}}r^{i-2}q^{-1}$,
for $i\leq m$ and $r^{n-3}=q\overline{q}$.
We are looking for $Q\in GL(n,\mathbb{K})$ and $C\in C(GL(n), r_{CG})$
such that $X=QJC$. We get $RJD=QJC$. By taking the conjugate,
we obtain $RJS\overline{D}=
QJS\overline{C}$, which implies
$SD^{-1}S\overline{D}=SC^{-1}S\overline{C}$.
Let $C=\mathrm{diag}(c_{1},\ldots ,c_{n})$ with $c_{i+1}=c_{2}^{i}c_{1}^{1-i}$ for all $i$. Therefore $c_{i}$ must fulfill the following system
$\overline{d_{i}}d_{n+1-i}^{-1}= \overline{c_{i}}c_{n+1-i}^{-1}$.
Equivalently, $\frac{\overline{c_{2}}^{i-1}c_{1}^{n-i-1}}{\overline{c_{1}}^{i-2}c_{2}^{n-i}}=\frac{q}{r^{i-2}}$ must hold for all $i$. By making a change of variables
$c_{1}=xy$, $c_{2}=y$, we immediately obtain
$x\overline{x}=r$ and
$x^{n-3}\overline{y}y^{-1}=q$. The first equation clearly has solution
in $\mathbb{K}[j]$. Since $q/x^{n-3}$ has norm 1,
Hilbert's Theorem 90 implies that there exists a solution $y\in \mathbb{K}[j]$
to the equation $\overline{y}/y=q/x^{n-3}$.
Thus we find a solution to the system which in turn provides us with a matrix
$C\in C(GL(n), r_{CG})$ that satisfies $SD^{-1}S\overline{D}=SC^{-1}S\overline{C}$.
Finally we note that if we let $Q=XC^{-1}J^{-1}$, then $Q\in GL(n,\mathbb{K})$
because of the way $C$ was chosen.

\end{proof}

The Cremmer--Gervais case can be further generalized. We call a triple
$(\Gamma_{1},\Gamma_{2}, \tau)$ \emph{generalized Cremmer--Gervais} if $\Gamma_{1}=\{\alpha_{1},\ldots ,\alpha_{k}\}$.  Without loss of generality, such a triple has one of the forms:

Type 1: $\Gamma_{1}=\{\alpha_{1},\ldots ,\alpha_{k}\}$,
$\Gamma_{2}=\{\alpha_{n-k},\ldots ,\alpha_{n-1}\}$
and $\tau(\alpha_{i})=\alpha_{n-k+i-1}$.

Type 2:  $\Gamma_{1}=\{\alpha_{1},\ldots ,\alpha_{k}\}$,
$\Gamma_{2}=\{\alpha_{n-k},\ldots ,\alpha_{n-1}\}$
and $\tau(\alpha_{i})=\alpha_{n-i}$.

Let us recall that a necessary condition for $\overline{Z}(SL(n),r)$ to be non-empty is that the corresponding admissible triple satisfies
$s(\Gamma_{1})=\Gamma_{2}$ and $s\tau=\tau^{-1}s$, where
$s$ is given by $s(\alpha_{i})=\alpha_{n-i}$
for all $i=1,\ldots ,n-1$. If the triple is generalized Cremmer--Gervais then
this condition is satisfied.

\begin{thm}
Let $r$ be a non-skewsymmetric $r$-matrix corresponding to a generalized Cremmer--Gervais triple $(\Gamma_{1},\Gamma_{2},\tau)$. Then $\overline{H}^{1}_{BD}(GL(n),r)$ consists of one element, the class of $J$.
\end{thm}

\begin{proof}
%Without loss of generality we may assume that
%$\Gamma_{1}=\{\alpha_{1},\ldots ,\alpha_{k}\}$
%and the admissible triple has one of the following forms:
First let us describe the centralizer $C(GL(n), r)$.

For type 1, i.e. $\Gamma_{1}=\{\alpha_{1},\ldots ,\alpha_{k}\}$,
$\Gamma_{2}=\{\alpha_{n-k},\ldots ,\alpha_{n-1}\}$
and $\tau(\alpha_{i})=\alpha_{n-k+i-1}$, the centralizer
$C(GL(n), r)$ consists of matrices $\mathrm{diag}(p_{1},\ldots ,p_{n})$
such that $p_{i-1}p_{i}^{-1}=p_{n-k+i-1}p_{n-k+i}^{-1}$ for all $i\leq k$.

For type 2, i.e. $\Gamma_{1}=\{\alpha_{1},\ldots ,\alpha_{k}\}$,
$\Gamma_{2}=\{\alpha_{n-k},\ldots ,\alpha_{n-1}\}$
and $\tau(\alpha_{i})=\alpha_{n-i}$, the corresponding $C(GL(n), r)$ consists of matrices $\mathrm{diag}(p_{1},\ldots ,p_{n})$
such that $p_{i}p_{i+1}^{-1}=p_{n-i}p_{n-i+1}^{-1}$ for all $i\leq k$.
We note that $k\leq [\frac{n-1}{2}]$, since otherwise $\tau$ has fixed points.

Let us assume that $X\in \overline{Z}(GL(n),r)$, for a triple
$(\Gamma_{1},\Gamma_{2},\tau)$ of the first type. Then $X=RJD$, where $R\in GL(n,\mathbb{K})$ and $D=\mathrm{diag}(d_{1},\ldots ,d_{n})$ is such that $SD^{-1}S\overline{D}\in C(GL(n), r)$. Let $p_{i}=\overline{d_{i}}d_{n+1-i}^{-1}$. Then $p_{n+1-i}=\overline{p_{i}}^{-1}$. On the other hand,
since $\mathrm{diag}(p_{1},\ldots ,p_{n})\in C(L(n), r)$, we have $p_{i-1}p_{i}^{-1}=p_{n-k+i-1}p_{n-k+i}^{-1}$ for all $i\leq k$. This further implies
$p_{i}p_{n-k+i}^{-1}=
p_{k-i+1}p_{n+1-i}^{-1}$ for all $i\leq k$. Thus we get
$p_{i}\overline{p}_{k-i+1}=p_{k-i+1}\overline{p_{i}}$, which
is equivalent to $p_{i}/p_{k-i+1}\in \mathbb{K}$. Equivalently,
$\frac{d_{i}d_{n+1-i}}{d_{k-i+1}d_{n-k+i}}\in \mathbb{K}$, for $i\leq k$.

Let us prove that $X$ is equivalent to $J$. For this, it is enough to
determine $C\in C(GL(n), r)$ which satisfies $SD^{-1}S\overline{D}=SC^{-1}S\overline{C}$. Let $C=\mathrm{diag}(c_{1},\ldots ,c_{n})$. The preceding condition is equivalent to the system: $\overline{c_{i}}c_{n+1-i}^{-1}=\overline{d_{i}}d_{n+1-i}^{-1}$, for $i\leq n$. On the other hand, since $C\in C(GL(n), r)$, $c_{i-1}c_{i}^{-1}=c_{n-k+i-1}c_{n-k+i}^{-1}$ for $i\leq k$. It follows that
$c_{i}c_{n-k+1}=c_{1}c_{n-k+i}$ and $c_{k-i+1}c_{n-k+1}=c_{1}c_{n-i+1}$. Consequently,
$c_{i}c_{n-i+1}=c_{k-i+1}c_{n-k+i}$. Furthermore,
$\frac{\overline{c}_{k-i+1}c_{k-i+1}}{\overline{c_{i}}c_{i}}=
\frac{\overline{d}_{k-i+1}d_{n+1-i}}{d_{n-k+i}\overline{d_{i}}}=:\lambda_{i}$. We note that $\lambda_{i}\in \mathbb{K}$ since $\frac{d_{i}d_{n+1-i}}{d_{k-i+1}d_{n-k+i}}\in \mathbb{K}$, for $i\leq k$. Thus we have obtained that the norm $c_{k-i+1}/c_{i}$ should be $\lambda_{i}$.  Now, if $c_{1}$,\ldots ,$c_{[\frac{k}{2}]}$ are fixed, then one can determine $c_{[\frac{k}{2}]+1}$,\ldots , $c_{k}$ since one can solve equations of the type
$x\overline{x}=\lambda_{i}$. The remaining unknown $c_{n-i+1}$ are determined by the relation $c_{k-i+1}c_{n-k+1}=c_{1}c_{n-i+1}$. Thus we have proved the existence of $C\in C(GL(n), r)$ and in conclusion $X$ and $J$ are equivalent.

Now let us consider $X\in \overline{Z}(SL(n),r)$, where the triple
$(\Gamma_{1},\Gamma_{2},\tau)$ is of the second type. Again we have a decomposition $X=RJD$, where $R\in GL(n,\mathbb{K})$ and $D=\mathrm{diag}(d_{1},\ldots ,d_{n})$ is such that $SD^{-1}S\overline{D}\in C(GL(n), r)$. Let $p_{i}=\overline{d_{i}}d_{n+1-i}^{-1}$. Since $\mathrm{diag}(p_{1},\ldots ,p_{n})\in C(GL(n), r)$, we have $p_{i}p_{i+1}^{-1}=p_{n-i}p_{n-i+1}^{-1}$ for all $i\leq k$. Since $p_{n+1-i}=\overline{p_{i}}^{-1}$, we easily get $p_{i}/p_{i+1}\in \mathbb{K}$, or equivalently, $\frac{d_{i}d_{n-i}}{d_{i+1}d_{n-i+1}}\in \mathbb{K}$ for $i\leq k$.

Let us show that $X$ is equivalent to $J$. As in the preceding case, the problem is reduced to solving the following
system: $\overline{c_{i}}c_{n+1-i}^{-1}=\overline{d_{i}}d_{n+1-i}^{-1}$, for $i\leq n$. On the other hand, since $C\in C(GL(n), r)$, $c_{i}c_{i+1}^{-1}=c_{n-i}c_{n-i+1}^{-1}$ for all $i\leq k$. We immediately get that
the norm of $c_{i}/c_{n-i}$ is $\lambda_{i}:=\frac{\overline{d_{i}}d_{i+1}}{\overline{d_{n-i}}d_{n+1-i}}$ which belongs to $\mathbb{K}$ since  $\frac{d_{i}d_{n-i}}{d_{i+1}d_{n-i+1}}\in \mathbb{K}$ for $i\leq k$. If we fix $c_{i}$ and solve
equations $x\overline{x}=\lambda_{i}$, we can determine $c_{n-i}$. The remaining unknown, $c_{k+1}$,\ldots ,$c_{n-k}$ can be arbitrarily chosen satisfying the
condition $\overline{c_{i}}c_{n+1-i}^{-1}=\overline{d_{i}}d_{n+1-i}^{-1}$. Thus $C$ exists and therefore the twisted cohomology set consists of the class of $J$.

\end{proof}

\section{Other gauge groups and conjectures}
%QQQ

\subsection{Computation of $H_{BD}^1 (SL(n), r_{BD})$}

The group $SL(n)$ is a subgroup of $GL(n)$ consisting of matrices with determinant one.
Let $H$ be the subgroup of diagonal matrices in $SL(n)$.
%We will denote $(d_1, \ldots , d_n):=\mathrm{diag}(d_1,\ldots , d_n)$. 
Positive roots are given by the formula $e^{\alpha_i}=d_id_{i+1}^{-1}$, where $\mathrm{diag}(d_1,\ldots , d_n)\in H$.
We will first prove the cohomology triviality for the Drinfeld-Jimbo $r$-matrix.

\begin{lem}
The Belavin-Drinfeld cohomology $H_{BD}^1( SL(n), r_{DJ})$ is trivial.
\end{lem}
\begin{proof}
Let $X \in Z^1 (SL(n), r_{DJ})$. We have $X=QD$,
where $Q \in GL(n, \mathbb{K})$, $D \in H(\overline{\mathbb{K}})$.
Then $D^{-1}\sigma(D) \in H(\mathbb{K})$ for any $\sigma$ in the absolute Galois group of $\mathbb{K}$.
Thus $\det D = k \in \mathbb{K}$. Let $D'=\mathrm{diag}(1,1, \ldots , k)$. Then $X=(QD')I(D'^{-1}D)$ is the desired decomposition, which provides
an equivalence between $X$ and  $ I$.
%$\Box$
\end{proof}
Given an $r$-matrix from the Belavin--Drinfeld list,
let $\tau: \Gamma_1\to\Gamma_2$ be the corresponding admissible triple for $sl(n)$.
Let $\alpha_{i_1}, \ldots , \alpha_{i_k}$ be a string for $\tau$, $\tau (\alpha_{i_p})=\alpha_{i_{p+1}}$.
If $\tau (\alpha_{i_p})$ is not defined, then anyway we define the corresponding string, which consists of one element $\{\alpha_{i_p}\}$ only.
Moreover, for
any Belavin-Drinfeld triple we will also consider a string $\{\alpha_n\}$ with weight $n$.
%Moreover, we will need an ``imaginary'' root and the corresponding string $\alpha_n$ defined by %$e^{\alpha_n}=d_n $.
For any string $S=\{\alpha_{i_1}, \ldots , \alpha_{i_k}\}$ of $\tau$, we define the weight of $S$ by $w_S=\sum_p i_p$.
%For any string $\{\alpha_{i_1}, \ldots , \alpha_{i_k}\}$ of $\tau$, we define the weight of the string by
%$w_k=\sum_p i_p$, where $k=i_1$.
Let $t_1, \ldots , t_n$ be the ends of the strings with weights $w_1, \ldots , w_n$. We note that some indices in  $w_1, \ldots , w_n$ are missing unless $\Gamma_1$ is an empty set and $w_n=n$ is always present. Let $N=GCD(w_1, \ldots , w_n)$.
%Let $w_1, \ldots , w_n$ be the weights of the strings of
%$\tau$.  We note that some indices in  $w_1, \ldots , w_n$ are missing unless $\Gamma_1$ is an empty set
%and $w_n=n$ is always present.

%Let also $t_1, \ldots , t_n$ be the ends of the strings (note that $t_n=n$) and $N=GCD(w_1, \ldots , w_n)$.
\begin{thm}
The number of elements of $H^1_{BD}(SL(n), r)$ is $N$. Each cohomology class contains a diagonal matrix $D=A_1A_2$,
where $A_2\in C(GL(n), r)$ and
$A_1\in\mathrm{diag}(n, \mathbb{K})$.
Two such diagonal matrices $D_1=A_1A_2$  and $D_2=B_1B_2$ are contained in the same class of $H_{BD}^1(SL(n), r)$
if and only if $\det(A_1)=\det(B_1)$
in ${\mathbb K}^*/(\mathbb{K}^*)^N $.
%A diagonal matrix $D=(d_1, \ldots , d_n)\in SL(n,\overline{\mathbb{K}})$ belongs to the class number
%$m,\ 0\leq m\leq N-1$ iff
%$$\hslash^m s_{t_1}^{w_1}s_{t_2}^{w_2} \ldots  s_{n}^{w_{t_n}}$$
%is the $N$'th power of the element of $\mathbb{K}$, here $s_i=e^{\alpha_i}(D)$.
\end{thm}
\begin{proof} Let $X\in SL(n,\overline{\mathbb{K}})$ be a representative of a cohomology class of $H^1_{BD} (SL(n), r)$.
Then we can find $Q\in SL(n,\mathbb{K})$ and a diagonal matrix $D$ such that $X=QD$. Therefore, $\det (D)=1$ and $X\sim D$.
Using the fact that $H^1_{BD} (GL(n), r)$ is trivial we can find a presentation $D=A_1 A_2$ such that
$A_1 $ is diagonal and has $\mathbb{K}$-entries while $A_2\in C(GL(n), r )$.

%The problem is that $D$ is equivalent to $A_1$
%in $H^1_{BD} (r, GL(n))$ but generally speaking not in $H^1_{BD} (r, SL(n))$. It follows that we have to find another
%diagonal matrix $D'_1 \in SL(n, \mathbb{K})$, such that it represents the same class in $H^1_{BD} (r, GL(n))$ as $D_1$.
%Clearly, the only way to do this is to multiply $D_1$ by an element from $C(r, GL(n))\cap GL(n,\mathbb{K})$.
%Let $K=(k_1,\ldots ,k_n), \ k_i\in\mathbb{K}$ be that element. We note that $\det (K)=s_{a_1}^{t_1}s_{t_2}^{n_2} \ldots  s_{a_s}^{n_s}$

Let two diagonal matrices $D_1=A_1A_2$ and $D_2=B_1B_2$ be equivalent. Then we have
$$
A_1B_1^{-1}C_1=A_2^{-1}B_2C_2, \ C_1\in \mathrm{diag}(n,\mathbb{K}), \ C_2\in C(SL(n), r).
$$
We see that $A_1B_1^{-1}C_1=A_2^{-1}B_2C_2=K\in C(r, GL(n))\cap GL(n,\mathbb{K})$. Then $A_1K^{-1}=B_1C_1^{-1}$, $D_1=(A_1K^{-1})(A_2K)$.
Since $\det(C_1)=\det(C_2)=1$, it follows that the class of $D_1$ uniquely defines $\det(A_1)$ in $\mathbb{K}^*$ modulo the subgroup
generated by determinants of elements of $C(r, GL(n))\cap GL(n,\mathbb{K})$.

Let $K=\mathrm{diag}(k_1,\ldots, k_n)\in C(r, GL(n))\cap GL(n,\mathbb{K})$. Then it is easy to check that
$\det(K)=s_{t_1}^{w_1}s_{t_2}^{w_2} \ldots  s_{t_n}^{w_n}$ (where $s_p=k_p/k_{p+1}, s_n=k_n$) is the $N$-th power of an element of $\mathbb{K}$.

Conversely, let $D=\mathrm{diag}(d_1,\ldots ,d_n)\in Z(SL(n), r)$ and $D=A_1A_2$ as above.
%Again, we have $X=QD$, for $Q \in SL(n, \mathbb{K})$, $D \in H(\overline{\mathbb{K}})$.
%Then $X \sim D$ and we can solve the problem for diagonal matrices.\\
It is sufficient to show that if $\det(A_1)=u^N$ for some $u\in\mathbb{K}^*$, then $D\sim I$.
There are integers $m_i$ so that $\sum m_iw_i=N$.
Set again $s_p=d_p/d_{p+1}, s_n=d_n$ and choose a string. If $t_p$ is the end of the string, set $s_i=s_p=u^{m_p}$
along the string. Solving the corresponding system for $\{d_i\}$, we find $d_1,d_2,\ldots,d_n\in\mathbb{K}$ (each $d_i$ will be a power of $u$),
such that the corresponding diagonal matrix $C=\mathrm{diag}(d_1,\ldots,d_n)$ has determinant $u^N$ and by construction $C\in C(r, GL(n))\cap GL(n,\mathbb{K})$.
Then $D=(A_1C^{-1})(CA_2)$ and $D\sim I$.

\end{proof}

\subsection{Computation of $\overline{H}_{BD}^1 (SL(n), r_{CG})$}

In this section we will compute Belavin-Drinfeld twisted cohomology for the Cremmer-Gervais $r$-matrix when the gauge group is $SL(n)$. The definition of this cohomology is exactly the same as in the $GL(n)$ case.
\begin{lem}
Any element of $\overline{Z}(SL(n), r_{CG})$ is equivalent to an element of the form $\alpha h_m J$, where $\alpha \in \overline{\mathbb{K}}$, $h_m=\mathrm{diag}(\hbar^m, 1, 1,\ldots, 1)$.
\end{lem}
\begin{proof}
By Proposition \ref{CG} an arbitrary cocycle can be written as $RJC$, where $R \in GL(n, \mathbb{K})$, $C \in C(GL(n), r_{CG})$. We can write $C=xC_1$, where $x \in \overline{\mathbb{K}}$, $C \in C(SL(n), r_{CG})$. Also we have $R=yh_mR_1$, where $y \in \mathbb{K}$, $R_1 \in C(SL(n), r_{CG})$. Therefore $RJC=R_1\alpha h_m J C_1 \sim \alpha h_m J$.
\end{proof}
\begin{lem}
If $\alpha_1 h_{m_1}J$ is equivalent to $\alpha_2h_{m_2}J$ then $m_2 \equiv m_1\ (\!\!\!\!\mod n/2)$ if $n$ is even and $m_2 \equiv m_1\ (\!\!\!\!\mod n)$ if $n$ is odd.
\end{lem}
\begin{proof}
The condition $\alpha_1 h_{m_1}J \sim \alpha_2h_{m_2}J$ is equivalent to $ \alpha_2h_{m_2}J = R\alpha_1 h_{m_1}J C $, where $R \in SL(n, \mathbb{K})$, $C \in C(SL(n), r_{CG})$. This in turn is equivalent to $h_{m_1}^{-1}Rh_{m_2}=JC_1J^{-1}$, where $C_1=\alpha_1\alpha_2^{-1}C \in C(GL(n), r_{CG})$. Since $h_m, R, J$ are defined over $\mathbb{K}[j]$, we have $C_1$ is defined over $K[j]$. Let $C_1=\mathrm{diag}(c_1,\ldots, c_n)$ (recall that all elements of $C(SL(n), r_{CG})$ are diagonal). Applying conjugation we get $JC_1J^{-1}=h_{m_1}^{-1}Rh_{m_2}=\overline{h_{m_1}^{-1}Rh_{m_2}}=JS\overline{C_1}SJ^{-1}$. Thus $S\overline{C_1}S=C_1$ i.e. $c_i=\overline{c_{n+1-i}}$. From the structure of centralizer we have $c_i/c_{i+1}=c_{n+1-(i+1)}/c_{n+1-i}$ so $c_{i}/c_{i+1}=\overline{c_{i+1}/c_i}$. It follows that the norms of all diagonal elements are equal to $\gamma \in \mathbb{K}$. If $n$ is odd then considering the central element we get that the norms of all diagonal elements are in fact equal to $\gamma^2$, for some $\gamma \in \mathbb{K}$. Finally we have $\hbar^{m_2-m_1}=\det(h_{m_1}^{-1}Rh_{m_2})=\det(JC_1J^{-1})=\gamma^k$, where $k=n/2$ for even $n$ and $k=n$ for odd $n$. The result follows.
\end{proof}
\begin{thm}
$\overline{H}_{BD}^1(SL(n), r_{CG})$ consists of $k$ elements where $k=n/2$ for even $n$ and $k=n$ for odd $n$.
\end{thm}
\begin{proof}
Note that if $X \in SL(n)$ commutes with all elements of the centralizer then if $A \sim B$ then $AX \sim BX$. Indeed from $A=RBC$ we get $AX=RBCX=RBXC$. Note that the matrices $h_m$ commute with the centralizer. Therefore, to prove the theorem we need to show that $\alpha h_kJ \sim \beta J$, for some scalars $\alpha, \beta$ (the scalars are defined uniquely in such a way that the cocycles are elements of $SL(n)$). We will now consider the cases of odd and even $n$ separately. 

Let $n$ be even. We need to find $R \in SL(n, \mathbb{K})$ and $C \in C(SL(n), r_{CG})$ such that $\alpha h_k J=\beta R J C$. Let's denote $C_1 = \beta \alpha^{-1}C \in C(GL(n), r_{CG})$. Then the equation becomes $h_k J = R J C_1$. Take $C_1=\mathrm{diag}(j, -j, j,\ldots, -j)$. Then $R=h_kJC_1^{-1}J^{-1}$. $\det R = 1$, $\overline{R} = h_kJS(-C_1)SJ^{-1}=h_kJC_1J^{-1}=R$. Therefore $R \in SL(n, \mathbb{K})$ and we are done.

Now assume $n$ is odd. Again we need to find $R \in SL(n, \mathbb{K})$ and $C \in C(SL(n), r_{CG})$ such that $\alpha h_k J=\beta R J C$. Let $C_1 = \beta \alpha^{-1}C \in C(GL(n), r_{CG})$. Then we get $h_kJ=RJC_1$. Take $C_1=\hbar$. Then $R=h_kJC_1^{-1}J^{-1}$, $\det R =1$. Finally $\overline{R}=h_kJSC_1^{-1}SJ^{-1}=R$.
\end{proof}

\subsection{Belavin--Drinfeld cohomology conjecture}

\begin{conj}
Let $\mathfrak{g}$ be a simple Lie algebra and $r_{DJ}$ the Drinfeld--Jimbo $r$-matrix. For any connected algebraic group $G$ which has $\mathfrak{g}$ as its Lie algebra, $H^{1}_{BD}(G, r_{DJ})$ is trivial.
\end{conj}

\subsection{Quantization conjecture}

Let $L$ be a finite dimensional Lie algebra over $\mathbb{C}$
and $\delta$ a Lie bialgebra structure on $L(\mathbb{K})$ such that
$\delta=0\ (\!\!\!\!\mod\hbar)$.

Let $(U_{\hbar}(L),\Delta_{\hbar})$ be
the corresponding quantum group, in other words the dequantization functor
$\widehat{Q}$ sends  $(U_{\hbar}(L),\Delta_{\hbar})$ to $(L(\mathbb{K}), \delta)$.
Let $G$ be a connected algebraic group with Lie algebra $L$. We assume that
$G $ acts on $L$ by the adjoint action.
Consider $G(\overline{\mathbb{K}})$. Let us define the centralizer $C(\overline{\mathbb{K}}, \delta)$.

\begin{defn}
The centralizer $C(\overline{\mathbb{K}}, \delta)$ consists of all
$X\in G(\overline{\mathbb{K}})$ such that for any $l\in L$,
\[(\mathrm{Ad}_{X}\otimes \mathrm{Ad}_{X})\delta(\mathrm{Ad}_{X}^{-1}(l))=\delta(l).\]
\end{defn}

\begin{defn}
We say that $X\in  G(\overline{\mathbb{K}})$ is a \emph{Belavin--Drinfeld cocycle} associated to $\delta$ if $\sigma(X)=XC$,
where $C\in C(\overline{\mathbb{K}}, \delta)$.

Two cocycles, associated to $\delta$, $X_{1}$ and $X_{2}$
are \emph{equivalent} if $X_{1}=QX_{2}C$, where $Q\in G(\mathbb{K})$ and $C\in C(\overline{\mathbb{K}}, \delta)$.

The set of equivalence classes will be denoted by $H^{1}_{BD}(G, \delta)$.
\end{defn}
Now let us define quantum Belavin--Drinfeld cohomology. The
quantum group $(U_{\hbar}(L),\Delta_{\hbar})$ is defined over $\mathbb{O}=\mathbb{C}[[\hbar]]$.
We extend the Hopf structures of $U_{\hbar}(L)$ to $U_{\hbar}(L,\mathbb{K})=U_{\hbar}(L)\otimes_{\mathbb{O}}\mathbb{K}$
and $U_{\hbar}(L,\overline{\mathbb{K}})=U_{\hbar}(L)\otimes_{\mathbb{K}}\overline{\mathbb{K}}$.
By abuse of notation, $\Delta_{\hbar}$ denotes all three comultiplications.

\begin{defn}
Let $P$ be an invertible element of $U_{\hbar}(L,\overline{\mathbb{K}})$. We say that it belongs to $C(U_{\hbar}(L),\Delta_{\hbar})$ if
\[(P\otimes P)\Delta_{\hbar}(P^{-1}aP)(P^{-1}\otimes P^{-1})=\Delta_{\hbar}(a)\]
for all $a\in U_{\hbar}(L)$.

\end{defn}
Denote \[F_{P}:=(P\otimes P)\Delta_{\hbar}(P^{-1})\in U_{\hbar}(L,\overline{\mathbb{K}})^{\otimes 2}.\]

\begin{defn}
$P$ is called a \emph{quantum Belavin--Drinfeld cocycle} if
for any $\sigma \in \Gal(\overline{\mathbb{K}}/\mathbb{K})$,
there exists $C\in C(U_{\hbar}(L),\Delta_{\hbar})$ such that
$\sigma(P)=PC$.

Two quantum cocycles $P_{1}$ and $P_{2}$ are \emph{equivalent} if $P_{2}=QP_{1}C$
where $Q$ is an invertible element of $U_{\hbar}(L,\mathbb{K})$
and $C\in C(U_{\hbar}(L),\Delta_{\hbar})$.
\end{defn}
\begin{rem}
On $U_{\hbar}(L)$ consider the comultiplications
 $\Delta_{\hbar,P_{1}}(a)=F_{P_{1}}\Delta_{\hbar}(a)F_{P_{1}}^{-1}$ and
$\Delta_{\hbar,P_{2}}(a)=F_{P_{2}}\Delta_{\hbar}(a)F_{P_{2}}^{-1}$.
Clearly, $\Delta_{\hbar,P_{2}}(a)=(Q\otimes Q)\Delta_{\hbar,P_{1}}(Q^{-1}aQ)\cdot (Q^{-1}\otimes Q^{-1})$.
Since $Q\in U_{\hbar}(L(\mathbb{K}))$, it is natural to call $\Delta_{\hbar,P_{1}}$ and $\Delta_{\hbar,P_{2}}$
$\mathbb{K}$--equivalent comultiplications on  $U_{\hbar}(L(\mathbb{K}))$.
\end{rem}

The set of equivalence classes of quantum Belavin--Drinfeld cocycles
associated to $\Delta_{\hbar}$ will be denoted by $H^{1}_{q-BD}(\Delta_{\hbar})$.

\begin{conj}
There is a natural correspondence between $H^{1}_{BD}(G, \delta)$ and
$H^{1}_{q-BD}(\Delta_{\hbar})$.
\end{conj}

\textbf{Acknowledgment.} The authors are grateful to V. Kac, P. Etingof
and V. Hinich for valuable suggestions.

\end{document}